\documentclass[11pt]{article}
\usepackage{bbm}
\usepackage{mathrsfs}
\usepackage{amsfonts}
\usepackage{amssymb}
\usepackage{amssymb,amsmath,amsthm, amsfonts}
\providecommand{\U}[1]{\protect \rule{.1in}{.1in}}
\newtheorem{theorem}{Theorem}[section]

\newtheorem{corollary}{Corollary}[section]

\newtheorem{definition}{Definition}[section]

\newtheorem{lemma}{Lemma}[section]

\newtheorem{proposition}{Proposition}[section]
\newtheorem{remark}{Remark}[section]
\newtheorem{example}{Example}[section]

\def\essinf{\mathop{\rm essinf}}
\def\sup{\mathop{\rm sup}}
\makeatletter
   
   \@addtoreset{equation}{section}
\makeatother

\begin{document}

\title{Representation of asymptotic values for nonexpansive
stochastic control systems \thanks{The work has been supported in part by the NSF of P.R.China (No. 11222110), Shandong Province (No. JQ201202), NSFC-RS (No. 11661130148), 111 Project (No. B12023).}}
\author{Juan Li,\,\, Nana Zhao\footnote{Corresponding author}\\
{\small School of Mathematics and Statistics, Shandong University, Weihai, Weihai 264209, P.~R.~China.}\\
{\small{\it E-mails: juanli@sdu.edu.cn, nnz0528@163.com.}}
\date{August 02, 2017}}
\maketitle
\begin{abstract}
In ergodic stochastic problems the limit of the value function $V_\lambda$ of the associated discounted cost functional with infinite time horizon is studied, when the discounted factor $\lambda$ tends to zero. These problems have been well studied in the literature and the used assumptions guarantee that the value function $\lambda V _\lambda$ converges uniformly to a constant as $\lambda\to 0$. The objective of this work consists in studying these problems under assumptions, namely, the nonexpansivity assumption, under which the limit function is not necessarily constant. Our discussion goes beyond the case of the stochastic control problem with infinite time horizon and discusses also $V_\lambda$ given by a Hamilton-Jacobi-Bellman equation of second order which is not necessarily associated with a stochastic control problem. On the other hand, the stochastic control case generalizes considerably earlier works by considering cost functionals defined through a backward stochastic differential equation with infinite time horizon and we give an explicit representation formula for the limit of $\lambda V_\lambda$, as $\lambda\to 0$.
\bigskip
\end{abstract}

\noindent \textbf{Keywords.}
Stochastic nonexpansivity condition; limit value; BSDE.

\noindent \textbf{AMS Subject classification:} 60H10; 60K35

\section{{\protect \large {Introduction}}}
In our paper we study the limit behaviour of the optimal value of a discounted cost functional with infinite time horizon as the discount factor $\lambda>0$ tends to zero. For this we consider a stochastic control system given by the controlled stochastic equation
\begin{equation}\label{0}
dX_t^{x,u}=b(X_t^{u,x},u_t)dt+\sigma(X_t^{x,u},u_t)dW_t,\, t\ge 0,\ \ X_0^{x,u}=x\in \mathbb{R}^N,
\end{equation}

\noindent driven by a Brownian motion $W$ and an admissible control $u\in{\cal U}$, i.e. a control process $u$ which is adapted with respect to the filtration $\mathbf{F}=({\cal F}_t)_{t\ge 0}$ generated by $W$ and completed by all null sets. As we are interested in the limit behaviour of the controlled system, as $t\rightarrow +\infty$, we have to add to the usual Lipschitz and growth conditions on the coefficients $\sigma$ and $b$ also assumptions guaranteeing that, for all the process $X^{t,u}$ takes all its values in a compact $\overline{\theta}(\subset R^N)$, for all $x\in\overline{\theta}$ and all $u\in{\cal U}$. The cost functional $\overline{Y}_0^{x,u}$ associated with the dynamics $X^{x,u}$ is defined through a backward stochastic differential equation (BSDE) on the infinite time interval $[0,+\infty)$:

\begin{equation}\label{1} \overline{Y}^{\lambda,x,u}_t=\overline{Y}^{\lambda,x,u}_T+\int_t^T(\psi(X_s^{x,u}, \overline{Z}^{\lambda,x,u}_s,u_s)-\lambda \overline{Y}^{\lambda,x,u}_s)ds-\int_t^T\overline{Z}^{\lambda,x,u}_sdW_s,\, 0\le t\le T<+\infty,\end{equation}

\noindent and we define the value function
\begin{equation}\label{3} V_\lambda(x):=\inf_{u\in{\cal U}}\overline{Y}_0^{x,u}.\end{equation}

We remark that, if $\psi(x,z,u)$ doesn't depend on $z$, we get the cost functional considered in \cite{Buckdahn 2013}:
\smallskip
\begin{equation}\label{2} \overline{Y}_0^{x,u}=E\left[\int_0^\infty e^{-\lambda t}\psi(X_t^{x,u},u_t)dt\right].\end{equation}

\smallskip

\noindent However, since the pioneering work by Pardoux and Peng \cite{peng1990} on BSDEs in 1990 and its extension by Darling, Pardoux \cite{Darling 1997} and by Peng \cite{S. Peng 1991}, and in particular since the works by Peng \cite{P 1992}, \cite{Peng 1997} on BSDE methods in stochastic control, it has become usual to study stochastic control systems whose cost functionals are defined through a BSDE. As concerns BSDEs with infinite time horizon, Chen \cite{Chen 1992} was the first to study such equations on an unbounded random time interval, Hamad\`{e}ne, Lepeltier and Wu \cite{Wu 1999} studied reflected BSDEs with one reflecting barrier and with infinite time horizon. Moreover, Briand, Hu \cite{Hu 1998} and Royer \cite{Royer 2004} generalized the existence results for BSDEs with unbounded random terminal time.

In our paper we begin our studies with the above infinite terminal time BSDE (\ref{1}), where we use techniques developed by Debusche, Hu and Tessitore \cite{Debussche 2011}, and we provide new estimates. Let us point out that in \cite{Debussche 2011} the authors have studied Ergodic BSDEs first introduced by Fuhrman, Hu and Tessotore \cite{Fuhrman 2009}; their $\lambda$ is a part of the solution. Our BSDE differs from theirs, its driving coefficient $\psi$ depends also on the control process, and its study differs, since we are not interested in the ergodic case, we study the limit behaviour of the value function $\lambda V_\lambda$ as $\lambda\rightarrow 0$ under assumptions which don't imply that the limit value function is a constant.

The limit problem for deterministic and stochastic control systems has been studied by different authors. Quincampoix and Renault \cite{Quincampoix 2011} studied a deterministic control problem with infinite time horizon and investigated the limit behaviour of the discounted value value function, when the discount factor tends to zero. For this they used a so-called nonexpansivity condition, and they gave, in particular, examples which show that -unlike the ergodic case- the limit value function can depend on the initial state $x$. In Buckdahn, Goreac and Quincampoix \cite{Buckdahn 2013} these studies are extended to stochastic control problems with value functions of the form (\ref{2}) (Abel mean) but also of Ces\'{a}ro mean. In \cite{Marc 2015}, for the case of deterministic controls, Cannarsa and Quincampoix extend these approaches by using a measurable viability theorem of Frankowska, Plaskacz, Rzezuchowski \cite{Frankowska 1995}, they characterize $V_\lambda$ as constrained viscosity solution of an associated Hamilton-Jacobi equation, and they study the limit problem.

The studies in our paper are heavily inspired by \cite{Buckdahn 2013} and \cite{Frankowska 1995}. The key assumption in \cite{Buckdahn 2013}, which allows to take the limit of the classical value function $\lambda V_\lambda$ (see (\ref{3})) as $\lambda\rightarrow 0$ is the nonexpansivity condition. However, as we generalize the cost functional by defining it through an infinite time horizon BSDE, we have also to extend our nonexpansivity assumption to the more general case we investigate (see our assumption (\ref{r49}) in Section 2). This extension is non trivial, it gives a stability to this assumption under Girsanov transformation which we have to work with, but however our condition coincides with that given in \cite{Buckdahn 2013}, if $\psi$ is independent of $z$. Under our nonexpansivity condition we show that the family of functions $\{\lambda V_\lambda\}$ is equi-continuous and equi-bounded on $\overline{\theta}$. Hence, due to the Arzel\`{a}-Ascoli Theorem, as $\lambda \rightarrow 0$, $\{\lambda V_\lambda\}$ has an accumulation point in the space of continuous functions over $\bar{\theta}$ endowed with the supremum norm.

The main objective of our paper is to get the existence of the limit, i.e., the uniqueness of this accumulation point, and to characterize the limit function $w_0=\lim_{\lambda\rightarrow 0}\lambda {V}_\lambda$. In our approach PDE methods play a central role. We recall that the PDE approach for the study of the limit behaviour for solutions of Hamilton-Jacobi equations with coercitive Hamiltonian essentially originates from Lions, Papanicolaou and Varadhan \cite{P.-L. Lions}. This work was extended by Arisawa \cite{M. Arisawa 1998} for the deterministic control setting and by Arisawa and Lions \cite{Arisawa Lions 1998} to the stochastic control framework. For subsequent works and extensions the reader is referred to \cite{Artstein 2000}, \cite{Quincampoix 2011} for the deterministic control case, and to \cite{G. K. Basak 1997}, \cite{V. Borkar 2007}, \cite{R. Buckdahn 2005}, \cite{A. Richou 2009} and the references therein for the stochastic framework. But all these approaches were made in the ergodic case, under suitable assumptions guaranteeing that the limit value is independent of the initial data.

In our paper we too use a PDE approach. For this end we characterize $V_\lambda$ as constrained viscosity solution of the associated Hamilton-Jabobi-Bellman (HJB) equation

\smallskip

\centerline{$\lambda V_\lambda(x)+H(x,DV_\lambda(x),D^2V_\lambda(x))=0,\, \, x\in\theta,$}

\centerline{$\lambda V_\lambda(x)+H(x,DV_\lambda(x),D^2V_\lambda(x))\ge 0,\, x\in\partial\theta,$}

\noindent (see Section 3).  Avoiding assumptions which lead to the ergodic case, we suppose that the Hamiltonian $H$ satisfies a radial monotonicity condition

$$H(x,lp,lA)\le H(x,p,A),\, l\ge 1,\, (p,A)\in \mathbb{R}^N\times {\cal S}^N,$$

\noindent where ${\cal S}^N$ denotes the set of symmetric $N\times N$ matrices. This condition was introduced in \cite{Marc 2015}, and it guarantees the monotone and uniform convergence of $\lambda V_\lambda$, as $\lambda\rightarrow 0.$
As this convergence result for the constrained solution $V_\lambda$ of the above HJB equation is not directly related with the characterization of $V_\lambda$ as value function of our stochastic control problem, by using Katsoulakis' comparison results \cite{Katsoulakis 1994} for constrained solutions of PDEs, we extend our discussion to more general Hamiltonians which are not necessarily related with a stochastic control problem, but which satisfy the radial monotonicity condition. For this general case we characterize the limit $w_0=\lim_{\lambda\rightarrow 0}\lambda V_\lambda$ as maximal viscosity subsolution of some limit HJB equation (Theorem \ref{the:3.4}). More precisely, we prove that
$$w_0(x)=\sup\{w(x):\, w\in \mbox{Lip}_{M_0}(\overline{\theta}),\, w+\overline{H}(x,Dw,D^2w)\le 0\mbox{ on }\theta \mbox{ in viscosity sense}\}$$

\noindent $x\in\overline{\theta},$ where $\overline{H}(x,p,A)=\min\left\{M_0,\sup_{l>0}H(x,lp,lA)\right\}$ (For details, see Theorem \ref{the:3.4}).

\noindent After, coming back to the special case that $V_\lambda$ is the value function of our stochastic control problem, we characterize the limit function $w_0=\lim_{\lambda\rightarrow 0}\lambda V_\lambda$ as viscosity solution by passing to the limit in the HJB equation associated with $V_\lambda$. For the special case $\psi(x,z,u)=\psi_1(x,u)+g(z)$ we give an explicit representation of $w_0$ (see Theorem \ref{th:4.2})  using Peng's notion of $g$-expectation $\varepsilon^g[\cdot]$ (\cite{peng}); it's a non linear expectation introduced through a BSDE with driving coefficient $g$. More precisely, we show that
$$ w_0(x)=\inf_{t\ge 0,u\in{\cal U}}\varepsilon^g [\min_{v\in U}\psi(X_t^{x,u},0,v) ],\, x\in \overline{\theta}.$$

Our paper is organized as follows. In Section 2 we present the basic assumptions on the coefficient functions $b, \sigma, \psi$, we define the value function $V_\lambda(x)$, and we prove the existence and the uniqueness of the solution of the BSDEs on the infinite time interval $[0,\infty)$ (Proposition \ref{th:2.4}). We introduce the stochastic nonexpansivity condition and show that the nonexpansivity condition combined with standard assumptions implies the stochastic nonexpansivity condition (Proposition \ref{p:2.1}). A consequence is that the family of functions $\{\lambda V_\lambda\}_{\lambda>0}$ is equicontinuous and equibounded on $\overline{\theta}$ (Lemma \ref{lem:2.6}). In Section 3 we first define the constrained viscosity solution of general HJB equations which are not necessarily related with a stochastic control problem, and then we show in this general framework that $\lambda V_\lambda$ is monotone and converges uniformly to some limit $w_0$ as $\lambda\rightarrow 0$ (Theorem \ref{th:3.3}). Moreover, we give an explicit representation of $w_0(x)$ (Theorem \ref{the:3.4}). In Section 4 we consider the Hamiltonian $H$ related to the stochastic control problem, and we characterize $V_\lambda$ as the unique viscosity solution on $\overline{\theta}$ of the associated HJB equation (Proposition \ref{th:3.3.1} and Proposition \ref{th:3.2}). For the convenience of the reader, we give the proof of the dynamic programming principle (DPP) in the Appendix. Moreover, still in the stochastic control case the HJB equation satisfied by $w_0(x)$ (Theorem \ref{th:4.1}) is studied and an explicit formula for $w_0(x)$ (Theorem \ref{th:4.2}) is given with the help of the $g$-expectation, a nonlinear expectation introduced by Peng in \cite{peng}.

\section{ {\protect \large Preliminaries}}
Let $\{W_t\}_{t\geq0}$ be a standard $d$-dimensional Brownian motion defined on a complete probability space $(\Omega,\mathcal{F},\mathbb{P})$. Let $\mathbb{F}=\{\mathcal{F}_t\}_{t\geq0}$ be the filtration generated by $\{W_t\}_{t\geq0}$, and augmented by all $\mathbb{P}$-null sets. We put $\mathcal{F}_\infty=\bigvee\limits_{t\geq0}\mathcal{F}_t$. For any $N\geq1$, $|x|$ denotes the Euclidean norm of $x\in\mathbb{R}^N$ and $\langle\cdot,\cdot\rangle$ denotes the Euclidean scalar product. We introduce the following spaces of stochastic processes:
\begin{equation*}
\begin{split}
  &S_{\mathbb{F}}^2(\mathbb{R}):=\Big\{(\phi_t)_{0\leq t< \infty}\ \text{real-valued continuous}\ \mathbb {F}\text{-adapted process}: \mathbb{E}[\sup\limits_{t\in[0,\infty)}|\phi_t|^2] <\infty\Big\};\\
 &\mathcal{H}_{\mathbb{F}}^2(\mathbb{R}^{d}):= \Big\{(\phi_t)_{0\leq t< \infty}\ \mathbb{R}^{d}\text{-valued}\ \mathbb{F}\text{-progressively measurable process}: \mathbb{E}[\int_0^\infty|\phi_t|^2dt] <\infty\Big\};\\
 &\mathcal{H}_{\mathbb{F}}^{2,-2\lambda}(0,T;\mathbb{R}^d):=\{(\phi_t)_{0\leq t\leq T}\ \mathbb{R}^d\text{-valued}\ \mathbb{F}\text{-progressively measurable process}: \\
&\qquad\qquad\qquad\qquad\quad \mathbb{E}[\int_0^T\exp(-2\lambda t)|\phi_t|^2dt]<\infty\};\\
 &L_{\mathbb{F}}^\infty(0,\infty;\mathbb{R}^d):=\{(\phi_t)_{0\leq t< \infty}\ \mathbb{R}^{d}\text{-valued}\ \mathbb {F}\text{-adapted\ essentially\ bounded\ process}\};\\
 &L^2(\mathcal{F}_\infty;\mathbb{R}):=\Big\{\xi\ \text{real-valued}\ \mathcal{F}_\infty \text{-measurable random variable}:\mathbb{E}[|\xi|^2]<\infty\Big\}.
 \end{split}
\end{equation*}

We suppose that $(U,d)$ is a compact metric space, $U$ is our control state space, and $\mathcal{U}=L_{\mathbb{F}}^\infty(0,\infty;U)$ is the space of all admissible control processes. It is defined as the set of all $U$-valued $\mathbb{F}$-adapted processes. Let us consider functions $b:\mathbb{R}^N\times U \rightarrow \mathbb{R}^N$ and $\sigma:\mathbb{R}^N\times U\rightarrow\mathbb{R}^{N\times d}$ satisfying standard conditions of continuity and Lipschitz property:
\begin{equation*}\label{r1}
  \left\{
\begin{array}{llll}
\mbox{(Hi)}\ b,\ \sigma\ \mbox{are\ uniformly\ continuous\ on}\ \mathbb{R}^N\times U,\\
\mbox{(Hii)}\ \text{There\ exists\ a\ constant}\ c>0\ \mbox{such\ that} \\
\ \ \ |b(x,u)-b(x',u)|+|\sigma(x,u)-\sigma(x',u)| \leq c|x-x'|,\ \mbox{for\ all}\ x,\ x'\in\mathbb{R}^N,\ u\in U,\\   \tag{H1}
\ \ \ |b(x,u)|+|\sigma(x,u)|\leq c(1+|x|),\ \mbox{for\ all}\ x\in\mathbb{R}^N,\ u\in U.
\end{array}
\right.
\end{equation*}
\begin{lemma}\label{l:2.1}
Under our standard assumptions (H1), for all control $u\in\mathcal{U}$, the controlled stochastic system
\begin{equation}\label{r2}
  \left\{
\begin{array}{ll}
dX_t^{x,u}=b(X_t^{x,u},u_t)dt+\sigma(X_t^{x,u},u_t)dW_t, \ \ t\geq 0, \\
X_0^{x,u}=x\in\mathbb{R}^N,
\end{array}
\right.
\end{equation}
has a unique $\mathbb{R}^N$-valued continuous, $\mathbb{F}$-adapted solution $X^{x,u}=(X^{x,u}_t)_{t\geq0}$. Moreover, for all $T>0$, and $k\geq2$, there is a constant $C_k(T)>0$ such that
\begin{equation*}
\begin{split}
  &\mathbb{E}[\sup\limits_{0\leq s\leq t}|X_s^{x,u}|^k]\leq C_k(T)(1+|x|^k),\\
  &\mathbb{E}[\sup\limits_{0\leq s\leq t}|X_s^{x,u}-X_s^{x',u}|^k]\leq C_k(T)|x-x'|^k,\ t\in[0,T], x,\ x'\in\mathbb{R}^N, u\in\mathcal{U}.
\end{split}
\end{equation*}
\end{lemma}
The above result on SDEs is by now well known; for its proof the readers can refer to Ikeda, Watanabe \cite[pp.166-168]{Ikeda 1989} or  Karatzas, Shreve \cite[pp.289-290]{Karatzas 1987}.

We suppose that there exists a non-empty open set $\theta\subset \mathbb{R}^N$ with compact closure $\overline{\theta}$  such that $\overline{\theta}$ is invariant with respect to the control system (\ref{r2}). Recall that the invariance of $\overline{\theta}$ is defined by the fact that, for all control process $u\in\mathcal{U}$, if $x\in\overline{\theta}$, also $X_t^{x,u}\in\overline{\theta}$, for all $t\geq0$, $\mathbb{P}$-a.s.

Given now a function $\psi:\mathbb{R}^N\times \mathbb{R}^d\times U\rightarrow \mathbb{R}$, for any $\lambda>0$, we consider the following BSDE on the infinite time interval $[0,\infty)$:
\begin{equation}\label{r40}
  \overline{Y}_t^{\lambda,x,u}=\overline{Y}_T^{\lambda,x,u}+\int_t^T(\psi(X_s^{x,u},\overline{Z}_s^{\lambda,x,u},u_s)-\lambda \overline{Y}_s^{\lambda,x,u})ds-\int_t^T\overline{Z}_s^{\lambda,x,u}dW_s,\ 0\leq t\leq T<\infty.
\end{equation}
\begin{definition}
A couple of processes $(\overline{Y}^{\lambda,x,u},\overline{Z}^{\lambda,x,u})$ is called a solution of BSDE (\ref{r40}) on the infinite time interval if $(\overline{Y}^{\lambda,x,u},\overline{Z}^{\lambda,x,u})$ satisfies the equation (\ref{r40}), and $\overline{Y}^{\lambda,x,u}=(\overline{Y}^{\lambda,x,u}_t)_{t\geq0}\in\mathcal{S}_{\mathbb{F}}^2(\mathbb{R})$ is bounded by some constant $\widetilde{M}$ and $\overline{Z}^{\lambda,x,u}=(\overline{Z}^{\lambda,x,u}_t)_{t\geq0}$ is in the space
$$\mathcal{H}_{loc}^2(\mathbb{R}^d)=\{(\phi_t)_{0\leq t<\infty}: \mathbb{R}^d\mbox{-valued}\ \mathbb{F}\mbox{-progressively measurable},\
\displaystyle\mathbb{E}[\int_0^T|\phi_t|^2dt]<+\infty,\ 0\leq T<\infty\}.$$
\end{definition}

We suppose that $\psi:\mathbb{R}^N\times \mathbb{R}^d\times U\rightarrow \mathbb{R}$ satisfies the following conditions:
\begin{equation*}\label{r48}
  \left\{
\begin{array}{llll}
\mbox{(Hiii)}\ \psi\ \text{is\ continuous\ on}\ \mathbb{R}^N\times\mathbb{R}^d\times U;\\
\mbox{(Hiv)}\ \text{There\ exist\ nonnegative\ constants}\ K_x, K_z\ \mbox{and}\ M \ \text{such\ that}\\
\ \ \ |\psi(x,z,u)-\psi(x',z',u)|\leq K_x|x-x'|+K_z|z-z'|,\\
\ \ \ |\psi(x,0,u)|\leq M, \ \ (x,x',z,z',u)\in \mathbb{R}^{2N}\times\mathbb{R}^{2d}\times U.\tag{H2}
\end{array}
\right.
\end{equation*}
The following proposition will be used frequently in what follows. We adapt the proof from \cite{Debussche 2011}, and also prove new estimates.
\begin{proposition}\label{th:2.4}
Under the assumptions (\ref{r1}) and (\ref{r48}), BSDE (\ref{r40}) on the infinite time interval $[0,\infty)$ has a unique solution $(\overline{Y}^{\lambda,x,u},\overline{Z}^{\lambda,x,u})\in L^\infty
_{\mathbb{F}}(0,T;\mathbb{R})\times\mathcal{H}^2_{loc}(\mathbb{R}^d)$. Moreover, we have
\begin{equation*}
  |\overline{Y}_t^{\lambda,x,u}|\leq\frac{M}{\lambda},\ t\geq0,\ \text{and}\ \mathbb{E}[\int_0^\infty|e^{-\lambda t}\overline{Z}_t^{\lambda,x,u}|^2dt]\leq2(\frac{M}{\lambda})^2(2+\frac{K_z^2}{\lambda}).
\end{equation*}
\end{proposition}
\begin{proof}
\textbf{Uniqueness}. Let $x\in\mathbb{R}^N \mbox{and}\ u\in \mathcal{U}$ be arbitrarily given. Suppose that $(\overline{Y}_t^{1,\lambda,x,u},\overline{Z}_t^{1,\lambda,x,u})_{t\geq0}$ and $(\overline{Y}_t^{2,\lambda,x,u},\overline{Z}_t^{2,\lambda,x,u})_{t\geq0}$ are two solutions of BSDE (\ref{r40}) such that $\overline{Y}^{1,\lambda,x,u}, \overline{Y}^{2,\lambda,x,u}$ are continuous and bounded and $\overline{Z}^{1,\lambda,x,u}, \overline{Z}^{2,\lambda,x,u}\in \mathcal{H}^2_{loc}(\mathbb{R}^d)$. Let us set $\widehat{Y}_t=\overline{Y}_t^{1,\lambda,x,u}-\overline{Y}_t^{2,\lambda,x,u}$ and $\widehat{Z}_t=\overline{Z}_t^{1,\lambda,x,u}-\overline{Z}_t^{2,\lambda,x,u}$, $t\geq0$. Then, $\widehat{Y}$ is continuous and $|\widehat{Y}|\leq\overline{M}$, for some constant $\overline{M}$.
We define
\begin{equation*}
  \gamma_s=\left\{
\begin{array}{lll}  &\frac{\psi(X_s^{x,u},\overline{Z}_s^{1,\lambda,x,u},u_s)-\psi(X_s^{x,u},\overline{Z}_s^{2,\lambda,x,u},u_s)}{|\widehat{Z}_s|^2}(\widehat{Z}_s)^*,\ \mbox{if}\  \widehat{Z}_s\neq 0;\\
& 0, \ \mbox{otherwise},
\end{array}\right.
\end{equation*}
and we notice that $|\gamma_s|\leq K_z,\ s\geq0$. Let $0\leq T<\infty$ be arbitrarily fixed. We define the probability $\mathbb{P}_T^\gamma$ on $(\Omega,\mathcal{F})$ by setting
\begin{equation*}
  \frac{d\mathbb{P}_T^\gamma}{d\mathbb{P}}=\exp\{\int_0^T\gamma_sdW_s-\frac{1}{2}\int_0^T|\gamma_s|^2ds\}.
\end{equation*}
Then, from Girsanov's theorem,
\begin{equation*}
\begin{split}
  \widehat{Y}_t=&\widehat{Y}_T+\int_t^T(\psi(X_s^{x,u},\overline{Z}_s^{1,\lambda,x,u},u_s)-\psi(X_s^{x,u},\overline{Z}_s^{2,\lambda,x,u},u_s))ds-\lambda\int_t^T\widehat{Y}_sds-\int_t^T\widehat{Z}_sdW_s\\
  =& \widehat{Y}_T-\lambda\int_t^T\widehat{Y}_sds-\int_t^T\widehat{Z}_s(dW_s-\gamma_sds)\\
  =&\widehat{Y}_T-\lambda\int_t^T\widehat{Y}_sds-\int_t^T\widehat{Z}_sdW^{\gamma,T}_s,\ t\in[0,T],
  \end{split}
\end{equation*}
where $\displaystyle W^{\gamma,T}_t=W_t-\int_0^t\gamma_sds,\ t\in[0,T],$ is an $(\mathbb{F},\mathbb{P}_T^\gamma)$-Brownian motion. Applying It\^{o}'s formula to $e^{-\lambda s}\widehat{Y}_s$, we get
\begin{equation*}
  e^{-\lambda T}\widehat{Y}_T-e^{-\lambda t}\widehat{Y}_t=\int_t^Te^{-\lambda s}\widehat{Z}_sdW_s^{\gamma,T},\ t\in[0,T].
\end{equation*}
From standard estimates we see that $\displaystyle (\int_0^te^{-\lambda s}\widehat{Z}_sdW_s^{\gamma,T})_{t\in[0,T]}$ is an $(\mathbb{F},\mathbb{P}_T^\gamma)$-martingale.
Thus, denoting by $\mathbb{E}^\gamma_T[\cdot\big|\mathcal{F}_t]$ the conditional expectation under $\mathbb{P}_T^\gamma$, it follows that
\begin{equation*}
  \widehat{Y}_t=\mathbb{E}^\gamma_T[\widehat{Y}_t\big|\mathcal{F}_t]=\mathbb{E}_T^\gamma[e^{-\lambda (T-t)}\widehat{Y}_T\big|\mathcal{F}_t]-\mathbb{E}_T^\gamma[\int_t^Te^{-\lambda (s-t)}\widehat{Z}_sdW_s^{\gamma,T}\big|\mathcal{F}_t]=\mathbb{E}_T^\gamma[e^{-\lambda (T-t)}\widehat{Y}_T\big|\mathcal{F}_t],\ t\in[0,T].
\end{equation*}
Recall that $|\widehat{Y}_s|\leq\overline{M},\ s\geq0$. Hence, $|\widehat{Y}_t|\leq e^{-\lambda(T-t)}\overline{M},\ 0\leq t\leq T<\infty$.
Finally, letting $T$ tend to infinity, we obtain that, for any $t\geq0$, $\widehat{Y}_t=0,\ \mathbb{P}\text{-}a.s.$, i.e., $\overline{Y}_t^{1,\lambda,x,u}=\overline{Y}_t^{2,\lambda,x,u},\ \text{for\ all}\ t\geq0,\ \mathbb{P}\text{-}a.s.$

\medskip

\textbf{Existence.} For arbitrarily given $x\in\mathbb{R}^N, u\in \mathcal{U}\ \mbox{and}\ n\geq1$, we define $(\overline{Y}_t^{n,\lambda,x,u},\overline{Z}_t^{n,\lambda,x,u})_{t\geq 0}\in\mathcal{S}_{\mathbb{F}}^2([0,n];\mathbb{R})\times\mathcal{H}_{\mathbb{F}}^2([0,n];\mathbb{R}^d)$ as the unique solution of the following BSDE:
\begin{equation}\label{r80}
  \overline{Y}_t^{n,\lambda,x,u}=\int_t^n(\psi(X_s^{x,u},\overline{Z}_s^{n,\lambda,x,u},u_s)-\lambda\overline{Y}_s^{n,\lambda,x,u})ds-\int_t^n\overline{Z}_s^{n,\lambda,x,u}dW_s,\ t\in[0,n].
\end{equation}
Then, from a classical result for BSDEs we get the existence and the uniqueness of the solution $(\overline{Y}^{n,\lambda,x,u},\overline{Z}^{n,\lambda,x,u})$ under the assumption (H2). Now we will give the proof in four steps.\\
\textbf{Step 1.} $(\overline{Y}_t^{n,\lambda,x,u})_{t\in[0,n]}$ is bounded, uniformly with respect to $n$.\\
\indent Indeed, by introducing the ${\mathbb{F}}$-adapted process
\begin{equation*}
  \gamma_s^n=\left\{
\begin{array}{lll}  &\frac{\psi(X_s^{x,u},\overline{Z}_s^{n,\lambda,x,u},u_s)-\psi(X_s^{x,u},0,u_s)}{|\overline{Z}_s^{n,\lambda,x,u}|^2}(\overline{Z}_s^{n,\lambda,x,u})^*, \ \mbox{if}\  \overline{Z}_s^{n,\lambda,x,u}\neq 0;\\
& 0, \ \mbox{otherwise},\ s\in[0,n],
\end{array}\right.
\end{equation*}
the above BSDE takes the form
\begin{equation*}
  \overline{Y}_t^{n,\lambda,x,u}=\int_t^n(\psi(X_s^{x,u},0,u_s)-\lambda\overline{Y}_s^{n,\lambda,x,u})ds-\int_t^n\overline{Z}_s^{n,\lambda,x,u}(dW_s-\gamma_s^nds), \ t\in[0,n].
\end{equation*}
As $\mid\gamma_s^n\mid\leq K_z,\ s\in[0,n],\ n\geq0,$ we know from the Girsanov Theorem that $\displaystyle W_t^n=W_t-\int_0^t\gamma_s^nds,\ t\in[0,n],$ is a Brownian motion under $\displaystyle d\mathbb{P}^n=\exp\{\int_0^n\gamma_s^ndW_s-\frac{1}{2}\int_0^n|\gamma_s^n|^2ds\}d\mathbb{P}$.

Consequently, applying It\^{o}'s formula to $e^{-\lambda t}\overline{Y}^{n,\lambda,x,u}_t,\ t\in[0,n]$, and taking the conditional expectation $\mathbb{E}^n[\cdot\big|\mathcal{F}_t]$ with respect to $\mathbb{P}^n$, we obtain
\begin{equation*}
  \overline{Y}_t^{n,\lambda,x,u}=\mathbb{E}^n[\int_t^ ne^{-\lambda(s-t)}\psi(X_s^{x,u},0,u_s)ds\big|\mathcal{F}_t],\ t\in[0,n].
\end{equation*}
Finally, as $|\psi(x',0,u')|\leq M,\ (x',u')\in\mathbb{R}^N\times U$, it follows that
\begin{equation*}
  |\overline{Y}_t^{n,\lambda,x,u}|\leq M\int_t^ne^{-\lambda(s-t)}ds\leq\frac{M}{\lambda},\  t\in[0,n],\ n\geq1.
\end{equation*}
Let us show now the second step.\\
\textbf{Step 2.} The sequence $(\overline{Y}_t^{n,\lambda,x,u})_{t\in[0,n]},\ n\geq1,$ converges uniformly on compacts, $\mathbb{P}$-a.s., as $n\rightarrow\infty$.

For $n,\ m\geq1$ with $n\geq m$, we define\\
$$  \gamma_s^{n,m}=\frac{\psi(X_s^{x,u},\overline{Z}_s^{n,\lambda,x,u},u_s)-\psi(X_s^{x,u},\overline{Z}_s^{m,\lambda,x,u},u_s)}
  {|\overline{Z}_s^{n,\lambda,x,u}-\overline{Z}_s^{m,\lambda,x,u}|^2}(\overline{Z}_s^{n,\lambda,x,u}-\overline{Z}_s^{m,\lambda,x,u})^*,\ \mbox{if}\  \overline{Z}_s^{n,\lambda,x,u}\neq\overline{Z}_s^{m,\lambda,x,u};$$
\noindent and $\gamma_s^{n,m}=0$, otherwise, $s\in[0,m]$. As $|\gamma_s^{n,m}|\leq K_z,\ s\in[0,m]$, we can use the Girsanov Theorem to introduce the probability measure $\displaystyle d\mathbb{P}^{n,m}=\exp\{\int_0^m\gamma_s^{n,m}dW_s-\frac{1}{2}\int_0^m|\gamma_s^{n,m}|^2ds\}d\mathbb{P},$ under which $\displaystyle W_t^{n,m}=W_t-\int_0^t\gamma_s^{n,m}ds,\ t\in[0,m],$ is an $\mathbb{F}$-Brownian motion. From (\ref{r80}) we have
\begin{equation*}
  \overline{Y}_t^{n,\lambda,x,u}-\overline{Y}_t^{m,\lambda,x,u}=e^{-\lambda(m-t)}\overline{Y}_m^{n,\lambda,x,u}+\int_t^me^{-\lambda (s-t)}(\overline{Z}_s^{n,\lambda,x,u}-\overline{Z}_s^{m,\lambda,x,u})dW_s^{n,m},\ t\in[0,m].
\end{equation*}
Consequently, considering that $|\overline{Y}_m^{n,\lambda,x,u}|\leq\frac{M}{\lambda}$ (Step1), by taking the conditional expectation under $\mathbb{P}^{n,m}$, we get
\begin{equation*}
  |\overline{Y}_t^{n,\lambda,x,u}-\overline{Y}_t^{m,\lambda,x,u}|=|\mathbb{E}^{n,m}[e^{-\lambda(m-t)}\overline{Y}_m^{n,\lambda,x,u}\big|\mathcal{F}_t]|\leq\frac{M}{\lambda}e^{-\lambda(m-t)},\ 0\leq t\leq m\leq n,
\end{equation*}
i.e., for all $T>0,\ n\geq m\geq T,$
\begin{equation}\label{t1}
  \sup\limits_{t\in[0,T]}|\overline{Y}_t^{n,\lambda,x,u}-\overline{Y}_t^{m,\lambda,x,u}|\leq \frac{M}{\lambda}e^{-\lambda(m-T)}\xrightarrow[ n\geq m\rightarrow\infty]{}
0,\ \mathbb{P}\mbox{-}a.s.
\end{equation}
Consequently, there is a continuous adapted process $\overline{Y}^{\lambda,x,u}=(\overline{Y}_t^{\lambda,x,u})_{t\geq0}$ to which $(\overline{Y}_{t\wedge n}^{n,\lambda,x,u})_{t\geq0}$ converges uniformly on compacts, $\mathbb{P}$-a.s. Moreover, $|\overline{Y}_t^{\lambda,x,u}|\leq\frac{M}{\lambda},\ t\geq0,\  \mathbb{P}$-a.s.\\
\textbf{Step 3.} There is a process $\overline{Z}^{\lambda,x,u}=(\overline{Z}_t^{\lambda,x,u})_{t\geq0}\in\mathcal{H}^2_{loc}(\mathbb{R}^d)$ such that, for all $T>0$,
\begin{equation*}
  \mathbb{E}[\int_0^T|\overline{Z}_t^{\lambda,x,u}-\overline{Z}_t^{n,\lambda,x,u}|^2dt]\xrightarrow[n\rightarrow\infty]{}
0.
\end{equation*}

Indeed, for $n\geq m\geq T$, we get from (\ref{r80}) and (\ref{t1})
\begin{equation*}
\begin{split}
  &\mathbb{E}[\int_0^T|\overline{Z}_t^{n,\lambda,x,u}-\overline{Z}_t^{m,\lambda,x,u}|^2dt]\\
  \leq&2\mathbb{E}[\int_0^T|\overline{Y}_s^{n,\lambda,x,u}-\overline{Y}_s^{m,\lambda,x,u}||\psi(X_s^{x,u},\overline{Z}_s^{n,\lambda,x,u},u_s)-\psi(X_s^{x,u},\overline{Z}_s^{m,\lambda,x,u},u_s)| ds]\\
  &+\mathbb{E}[|\overline{Y}_T^{n,\lambda,x,u}-\overline{Y}_T^{m,\lambda,x,u}|^2]\leq (4K_z^2T+1)\frac{M^2}{\lambda^2}e^{-2\lambda(m-T)}+\frac{1}{2}\mathbb{E}[\int_0^T|\overline{Z}_t^{n,\lambda,x,u}-\overline{Z}_t^{m,\lambda,x,u}|^2dt].
\end{split}
\end{equation*}
This proves that, for $n\geq m\geq T,$
\begin{equation*}
 \mathbb{E}[\int_0^T|\overline{Z}_t^{n,\lambda,x,u}-\overline{Z}_t^{m,\lambda,x,u}|^2dt]\leq2(4K_z^2T+1)\frac{M^2}{\lambda^2}e^{-2\lambda(m-T)}\xrightarrow[n\geq m\rightarrow\infty]{}0.
\end{equation*}
This completes Step 3.\\
\textbf{Step 4.} Finally, recall that from (\ref{r80}), for $n\geq T$,
\begin{equation*}
  \overline{Y}_t^{n,\lambda,x,u}= \overline{Y}_T^{n,\lambda,x,u}+\int_t^T(\psi(X_s^{x,u},\overline{Z}_s^{n,\lambda,x,u},u_s)-\lambda\overline{Y}_s^{n,\lambda,x,u})ds-\int_t^T\overline{Z}_s^{n,\lambda,x,u}dW_s,\ t\in[0,T].
\end{equation*}
The Steps 2 and 3 allow to take the limit in this BSDE, as $n\rightarrow\infty$, and we obtain that $(\overline{Y}^{\lambda,x,u},\overline{Z}^{\lambda,x,u})$ is the solution of the following BSDE:
\begin{equation}\label{r81}
  \overline{Y}_t^{\lambda,x,u}=\overline{Y}_T^{\lambda,x,u}+\int_t^T(\psi(X_s^{x,u},\overline{Z}_s^{\lambda,x,u},u_s)-\lambda\overline{Y}_s^{\lambda,x,u})ds-\int_t^T\overline{Z}_s^{\lambda,x,u}dW_s,\ 0\leq t\leq T<+\infty,
\end{equation}
with $| \overline{Y}_t^{\lambda,x,u}|\leq\frac{M}{\lambda},\ t\geq0,$ and $\overline{Z}^{\lambda,x,u}\in\mathcal{H}^2_{loc}(\mathbb{R}^d)$.

It remains to show that
\begin{equation*}
\mathbb{E}[\int_0^\infty|e^{-\lambda t}\overline{Z}_t^{\lambda,x,u}|^2dt]\leq2(\frac{M}{\lambda})^2(2+\frac{K_z^2}{\lambda}).
\end{equation*}
For this, by applying It\^{o}'s formula to $|e^{-\lambda t}\overline{Y}_t^{\lambda,x,u}|^2$, it follows from (\ref{r81}) that, for all $T>0$,
\begin{equation*}
  \begin{split}
  &\mathbb{E}[\int_0^T|e^{-\lambda t}\overline{Z}_t^{\lambda,x,u}|^2dt]\leq\mathbb{E}[|e^{-\lambda T}\overline{Y}_T^{\lambda,x,u}|^2+2\int_0^Te^{-2\lambda t}\overline{Y}_t^{\lambda,x,u}\psi(X_t^{x,u},\overline{Z}_t^{\lambda,x,u},u_t)dt]\\
  \leq&\mathbb{E}[|e^{-\lambda T}\overline{Y}_T^{\lambda,x,u}|^2+2\int_0^Te^{-2\lambda t}|\overline{Y}_t^{\lambda,x,u}|(M+K_z|\overline{Z}_t^{\lambda,x,u}|)dt],
  \end{split}
\end{equation*}
and, hence,
\begin{equation*}
  \frac{1}{2}\mathbb{E}[\int_0^T|e^{-\lambda t}\overline{Z}_t^{\lambda,x,u}|^2dt]\leq\mathbb{E}[|e^{-\lambda T}\overline{Y}_T^{\lambda,x,u}|^2]+2\mathbb{E}[\int_0^Te^{-2\lambda t}M|\overline{Y}_t^{\lambda,x,u}|dt]+2K_z^2\mathbb{E}[\int_0^T|e^{-\lambda t}\overline{Y}_t^{\lambda,x,u}|^2dt].
\end{equation*}
Therefore, using that $|\overline{Y}_t^{\lambda,x,u}|\leq\frac{M}{\lambda},\ t\geq0$, we obtain the stated estimate for $\displaystyle \mathbb{E}[\int_0^\infty|e^{-\lambda t}\overline{Z}_t^{\lambda,x,u}|^2dt]$.
\end{proof}
For any $\lambda>0$, let us define the value function
\begin{equation}\label{r94}
  V_\lambda(x):=\inf\limits_{u\in\mathcal{U}}\overline{Y}_0^{\lambda,x,u},\ x\in\mathbb{R}^N,
\end{equation}
\noindent where $\overline{Y}^{\lambda,x,u}$ is introduced by the BSDE (\ref{r40}).

We make the following so called \underline{nonexpansivity condition} for (\ref{r94}): For all $x,x'\in \mathbb{R}^N, u\in U,$ there exists $v\in U$ such that, for all $z\in\mathbb{R}^d$,
\begin{equation*}\label{r49}
\left\{
\begin{array}{llll}
\mbox{(i)}\ g(x,x',u,v):=\langle x-x',b(x,u)-b(x',v)\rangle+\frac{1}{2}|\sigma(x,u)-\sigma(x',v)|^2\\
 \qquad\qquad\qquad\qquad+K_z|\sigma(x,u)-\sigma(x',v)||x-x'|\leq0;\\
\mbox{(ii)}\ \text{There\ exists\ a\ constant}\ \overline{c}_0>0\ \text{such\ that}\\
\ \ \ \ \  \widetilde{\psi}(x,x',z,u,v):=|\psi(x,z,u)-\psi(x',z,v)|-\overline{c}_0|x-x'|\leq0,   \tag{H3}
\end{array}
\right.
\end{equation*}
with $K_z>0$ introduced in (\ref{r48}).

We also introduce a new \underline{stochastic nonexpansivity condition:} For all $\varepsilon>0,\ \lambda>0,\ x,\ x'\in \overline{\theta}$, and all $u\in\mathcal{U}$, there exists $v\in\mathcal{U}$ such that, for all $\gamma\in L_{\mathbb{F}}^\infty(0,\infty;\mathbb{R}^d)$ with $|\gamma_s|\leq K_z$, dsdP-a.e., and with the notation $\displaystyle L_t^\gamma=\exp\{\int_0^t\gamma_sdW_s-\frac{1}{2}\int_0^t|\gamma_s|^2ds\}$,
\begin{equation*}\label{r50}
\left\{
\begin{array}{llll}
\mbox{(i)}\ \displaystyle\big(\mathbb{E}[L_t^\gamma|X_t^{x,u}-X_t^{x',v}|^2]\big)^{\frac{1}{2}}\leq|x-x'|+\varepsilon,\ t\geq0;\\
\mbox{(ii)}\ \displaystyle\lambda\int_0^\infty e^{-\lambda s}\mathbb{E}[L_s^\gamma|\psi(X_s^{x,u},\overline{Z}_s^{\lambda,x,u},u_s)-\psi(X_s^{x',v},\overline{Z}_s^{\lambda,x,u},v_s)|]ds\leq \overline{c}_0|x-x'|+\varepsilon. \tag{H4}
 \end{array}
\right.
\end{equation*}
\begin{remark}
Let us recall the nonexpansivity condition in \cite{Buckdahn 2013}, established for $\psi=\psi(x,u)$, which is extended by (\ref{r49}): For all $(x,x',u,v)\in\mathbb{R}^{2N}\times U^2$,
\begin{equation*}
  \sup\limits_{u\in U}\inf\limits_{v\in U}\max\big((\langle x-x',b(x,u)-b(x',v)\rangle+\frac{1}{2}|\sigma(x,u)-\sigma(x',v)|^2),|\psi(x,u)-\psi(x',v)|-\overline{c}_0|x-x'|\big)\leq 0.
\end{equation*}
Observe that, if $\psi$ is independent of $z$ (that is, $\psi=\psi(x,u)$), then $K_z=0$ and (\ref{r49}) coincides with the above nonexpansivity condition in \cite{Buckdahn 2013}.\\
\indent But (\ref{r50}) is new, it reformulates the stochastic nonexpansivity condition in \cite{Buckdahn 2013} by taking into account the BSDE over the infinite time interval $[0,\infty)$.
\end{remark}
\begin{example}
Let $d=1$ and $b(x,u)=-3x,\ \sigma(x,u)=x,\ \psi(x,u,z)=z$, for $x\in\mathbb{R}^N,\ u\in U$ and $z\in\mathbb{R}$. Then $K_z=1$, and for $\overline{c}_0=1$ we have
\begin{equation*}
\begin{split}
  g(x,x',u,v):=&\langle x-x',b(x,u)-b(x',v)\rangle+\frac{1}{2}|\sigma(x,u)-\sigma(x',v)|^2+K_z|\sigma(x,u)-\sigma(x',v)||x-x'|\\
  =&-\frac{3}{2}|x-x'|^2\leq0,
  \end{split}
\end{equation*}
and
\begin{equation*}
  \widetilde{\psi}(x,x',z,u,v):=|\psi(x,z,u)-\psi(x',z,v)|-\overline{c}_0|x-x'|=-\overline{c}_0|x-x'|\leq0.
\end{equation*}

\end{example}
\begin{proposition}\label{p:2.1}
Under the assumptions (\ref{r1}) and (\ref{r48}) the nonexpansivity condition (\ref{r49}) implies the stochastic nonexpansivity condition (\ref{r50}).
\end{proposition}
\begin{proof}
We fix arbitrarily $(x,x')\in \overline{\theta}^2$, $\lambda>0$, $T>0$, $\varepsilon>0$, and $u\in\mathcal{U}$. Without loss of generality, let us suppose that $u$ is a step process, i.e., that there exists a partition of $[0,T ]$, denoted by $0=t_0<t_1<t_2<\cdot\cdot\cdot<t_M=T$, and random variables $u_i\in L^0(\mathcal{F}_{t_i}; U)$, $0\leq i\leq M-1$, such that
\begin{equation*}
u=\sum^{M-1}_{i=0}u_i1_{(t_i,t_{i+1}]}.
\end{equation*}
The reader can be referred to \cite{Krylov 1999} for further details. Indeed, we can make this choice, since these step functions are dense in the space of admissible controls $\mathcal{U}$ endowed with the metric $\displaystyle(\mathbb{E}[\int_0^\infty e^{-t}d(u_t,u'_t)^2dt])^{\frac{1}{2}},\ u,\ u'\in \mathcal{U}$, and the controlled state process $X^{x,u}$ as well as the solution $(\overline{Y}^{\lambda,x,u},\overline{Z}^{\lambda,x,u})$ of the BSDE (\ref{r40}) are $L^2$-continuous in $u\in \mathcal{U}$.
Now we introduce the set-valued function
\begin{equation*}
\begin{split}
  \overline{\theta}^2\times U\ni(x,x',u)\rightsquigarrow\Xi(x,x',u):=&\{v\in U:g(x,x',u,v)\leq0,\ \widetilde{\psi}(x,x',z,u,v)\leq0,\ \text{for\ all}\ z\in\mathbb{R}^d\}.
  \end{split}
\end{equation*}
From the fact that $\Xi$ is upper semicontinuous and  has nonempty compact values we know that there exists a Borel function (see Aubin and Frankowska \cite{Aubin 1990})
\begin{equation*}
  \widehat{v}:\overline{\theta}^2\times U\to U,\ \text{with}\ \widehat{v}(x,x',u)\in\Xi(x,x',u),\ \text{for\ all}\ (x,x',u)\in\overline{\theta}^2\times U.
\end{equation*}
\textbf{Step1.} On $[0,t_1]$, setting $\tau_0=0$, we define $$v_t^{0,0}:=\widehat{v}(X_0^{x,u},x',u_t)=\widehat{v}(x,x',u_0)(=v_0^{0,0}),$$
 and
\begin{equation*}
\begin{split}
 \tau_1:=&\inf\{t\geq0:g(X_t^{x,u},X_t^{x',v^{0,0}},u_t,v_t^{0,0})>\delta\ \mbox{or} \sup\limits_{z\in\mathbb{R}^d}\widetilde{\psi}(X_t^{x,u},X_t^{x',v^{0,0}},z,u_t,v_t^{0,0})>\delta\}\wedge\frac{t_1}{n},\  n\geq1,
 \end{split}
\end{equation*}
where $\delta>0$ is arbitrarily small and will be specified later.
Similar to the proof of Lemma 3 in \cite{Buckdahn 2013}, from the assumption that the compact $\overline{\theta}$ is invariant with respect to control system (\ref{r2}) and from (\ref{r1}) we get for all $t\in[0,t_1]$,
\begin{equation*}
\begin{split}
  &g(X_t^{x,u},X_t^{x',v^{0,0}},u_t,v_t^{0,0})
  =\langle X_t^{x,u}-X_t^{x',v^{0,0}},b(X_t^{x,u},u_t)-b(X_t^{x',v^{0,0}},v_t^{0,0})\rangle\\
  &+\frac{1}{2}|\sigma(X_t^{x,u},u_t)-\sigma(X_t^{x',v^{0,0}},v_t^{0,0})|^2+K_z|\sigma(X_t^{x,u},u_t)-\sigma(X_t^{x',v^{0,0}},v_t^{0,0})||X_t^{x,u}-X_t^{x',v^{0,0}}|\\
  \leq&\langle x-x',b(x,u_0)-b(x',v_0^{0,0})\rangle+\frac{1}{2}|\sigma(x,u_0)-\sigma(x',v_0^{0,0})|^2+K_z|\sigma(x,u_0)-\sigma(x',v_0^{0,0})||x-x'|\\
  &+c(|x-X_t^{x,u}|+|x'-X_t^{x',v^{0,0}}|)\\
  =&g(x,x',u_0,v_0^{0,0})+c(|x-X_t^{x,u}|+|x'-X_t^{x',v^{0,0}}|),
  \end{split}
\end{equation*}
and
\begin{equation*}
\begin{split}
&\widetilde{\psi}(X_t^{x,u},X_t^{x',v^{0,0}},\overline{Z}_t^{\lambda,x,u},u_t,v_t^{0,0})\\
=&|\psi(X_t^{x,u},\overline{Z}_t^{\lambda,x,u},u_t)-\psi(X_t^{x',v^{0,0}},\overline{Z}_t^{\lambda,x,u},v_t^{0,0})|-\overline{c}_0| X_t^{x,u}-X_t^{x',v^{0,0}}|\\
\leq&|\psi(x,\overline{Z}_t^{\lambda,x,u},u_0)-\psi(x',\overline{Z}_t^{\lambda,x,u},v_0^{0,0})|-\overline{c}_0|x-x'|+c(|X_t^{x,u}-x|+| X_t^{x',v^{0,0}}-x'|)\\
=&\widetilde{\psi}(x,x',\overline{Z}_t^{\lambda,x,u},u_0,v_0^{0,0})+c(|X_t^{x,u}-x|+|X_t^{x',v^{0,0}}-x'|),
\end{split}
\end{equation*}
for some constant $c$ depending on the coefficients $\sigma,b,\psi$ and on $\overline{\theta}$.

Then, from the choice of $v^{0,0}$ we have that
 \begin{equation*}
   g(X_t^{x,u},X_t^{x',v^{0,0}},u_t,v_t^{0,0})\leq c(|x-X_t^{x,u}|+|x'-X_t^{x',v^{0,0}}|),
 \end{equation*}
and
\begin{equation*}
  \widetilde{\psi}(X_t^{x,u},X_t^{x',v^{0,0}},\overline{Z}_t^{\lambda,x,u},u_t,v_t^{0,0})\leq c(|X_t^{x,u}-x|+|X_t^{x',v^{0,0}}-x'|),\  t\in[0,t_1].
\end{equation*}
 Thus, applying Markov's inequality and Burkholder's inequality, we have that, for all $p>1,\ n\geq1$, there is a constant $c_p>0$ such that
 \begin{equation}\label{r88}
 \begin{split}
   &\mathbb{P}(\tau_1<\frac{t_1}{n})\leq\mathbb{P}(\sup\limits_{t\in[0,\frac{t_1}{n}]}c(|x'-X_t^{x',v^{0,0}}|+|x-X_t^{x,u}|)\geq\delta)\\
   &\leq\frac{c}{\delta^{4p}}(\mathbb{E}[\sup\limits_{t\in[0,\frac{t_1}{n}]}|x'-X_t^{x',v^{0,0}}|^{4p}]+\mathbb{E}[\sup\limits_{t\in[0,\frac{t_1}{n}]}|x-X_t^{x,u}|^{4p}])\leq\frac{c^2_pt_1^{2p}}{\delta^{4p}n^{2p}}.
 \end{split}
 \end{equation}
 Recalling the definition of $L^\gamma$, we conclude that
 \begin{equation}\label{r89}
 \begin{split}
   \sup\limits_{|\gamma|\leq K_z}\mathbb{E}[L_{\frac{t_1}{n}}^{\gamma}1_{\{\tau_1<\frac{t_1}{n}\}}]\leq \sup\limits_{|\gamma|\leq K_z}\big(\mathbb{E}[(L_{\frac{t_1}{n}}^{\gamma})^2]\big)^\frac{1}{2}\big(\mathbb{P}(\tau_1<\frac{t_1}{n})\big)^\frac{1}{2}\leq e^{\frac{1}{2}K_z^2\frac{t_1}{n}}\frac{c_pt_1^p}{\delta^{2p}n^p}.
   \end{split}
 \end{equation}
For $1\leq i\leq n-1$, let us define iteratively $v^{0,i}$ and $\tau_{i+1}$. Given $\tau_i$ and $v^{0,i-1}\in\mathcal{U}$ we put
 \begin{equation*}
   v_t^{0,i}:=v_t^{0,i-1}1_{\{t\leq\tau_i\}}+\widehat{v}(X_{\tau_i}^{x,u},X_{\tau_i}^{x',v^{0,i-1}},u_t)1_{\{t>\tau_i\}},
 \end{equation*}
 and
 \begin{equation*}
 \begin{split}
   \tau_{i+1}:=&\inf\{t\geq\tau_{i}:g(X_t^{x,u},X_t^{x',v^{0,i}},u_t,v_t^{0,i})>\delta\ \mbox{or} \sup\limits_{z\in\mathbb{R}^d}\widetilde{\psi}(X_t^{x,u},X_t^{x',v^{0,i}},z,u_t,v_t^{0,i})>\delta\}\wedge\frac{(i+1)t_1}{n}.
   \end{split}
 \end{equation*}
 From the strong Markov property we have, in analogy to (\ref{r88}),
 \begin{equation*}
   \mathbb{P}(\tau_{i+1}-\tau_i<\frac{t_1}{n}/X_{\tau_i}^{x,u}=\widehat{x},X_{\tau_i}^{x',v^{0,i}}=\widehat{x'},\tau_i=\widehat{t}\ )\leq\frac{c^2_pt_1^{2p}}{\delta^{4p}n^{2p}},\ (\widehat{x},\widehat{x}')\in\overline{\theta}^2,\ \widehat{t}\in[0,\frac{it_1}{n}],
 \end{equation*}
 and, thus,
 \begin{equation*}
   \mathbb{P}(\tau_{i+1}-\tau_i<\frac{t_1}{n})\leq\frac{c^2_pt_1^{2p}}{\delta^{4p}n^{2p}}.
 \end{equation*}
Moreover, similar to (\ref{r89}) we get
 \begin{equation*}
 \begin{split}
    \sup\limits_{|\gamma|\leq K_z}\mathbb{E}[L_{t_1}^{\gamma}1_{\{\tau_{i+1}-\tau_i<\frac{t_1}{n}\}}]\leq \sup\limits_{|\gamma|\leq K_z}\big(\mathbb{E}[(L_{t_1}^{\gamma})^2]\big)^\frac{1}{2}\big(\mathbb{P}(\tau_{i+1}-\tau_i<\frac{t_1}{n})\big)^\frac{1}{2}\leq e^{\frac{1}{2}K_z^2t_1}\frac{c_pt_1^p}{\delta^{2p}n^p}.
   \end{split}
 \end{equation*}
This shows that there exists a constant $\overline{c}_p>0$ such that
\begin{equation}\label{112}
 \sup\limits_{|\gamma|\leq K_z}\mathbb{E}[L_{t_1}^{\gamma}1_{\{\tau_n<t_1\}}]\leq \sum_{i=0}^{n-1}\sup\limits_{|\gamma|\leq K_z}\mathbb{E}[L_{t_1}^{\gamma}1_{\{\tau_{i+1}-\tau_i<\frac{t_1}{n}\}}]\leq e^{\frac{1}{2}K_z^2t_1}\frac{\overline{c}_pt_1^p}{\delta^{2p}n^{p-1}}.
\end{equation}
Let $d\mathbb{P}^\gamma_{t_1}=L_{t_1}^\gamma d\mathbb{P}$, and recall that due to the Girsanov Theorem $\displaystyle W_t^\gamma=W_t-\int_0^t\gamma_sds,\  t\in[0,t_1],$ is an $(\mathbb{F},\mathbb{P}_{t_1}^\gamma)$-Brownian motion.
Let us define $\displaystyle\mathbb{E}_{t_1}^\gamma[\cdot]=\int_\Omega(\cdot) d\mathbb{P}_{t_1}^\gamma=\mathbb{E}[L_{t_1}^\gamma(\cdot)]$. Applying It\^{o}'s formula to $|X_t^{x,u}-X_t^{x',v^{0,n}}|^2$, for all $t\leq t_1$ we have
\begin{eqnarray}\label{r41}
\begin{split}
  \mathbb{E}_{t_1}^\gamma[|X_t^{x,u}-X_t^{x',v^{0,n}}|^2] =&|x-x'|^2 +2\mathbb{E}_{t_1}^\gamma\Big[\int_0^t\Big(\langle X_s^{x,u}-X_s^{x',v^{0,n}},b(X_s^{x,u},u_s)-b(X_s^{x',v^{0,n}},v_s^{0,n})\rangle\\
   &+ \frac{1}{2}|\sigma(X_s^{x,u},u_s)-\sigma(X_s^{x',v^{0,n}},v_s^{0,n})|^2\Big) ds \Big] \\ +&2\mathbb{E}_{t_1}^\gamma\Big[\int_0^t(X_s^{x,u}-X_s^{x',v^{0,n}})(\sigma(X_s^{x,u},u_s)-\sigma(X_s^{x',v^{0,n}},v_s^{0,n}))dW_s].
   \end{split}
\end{eqnarray}
Thus, substituting $dW_s=dW_s^\gamma+\gamma_sds$ and taking into account that $|\gamma_s|\leq K_z$, dsd$\mathbb{P}$-a.e., we obtain
\begin{equation}\label{r91}
\begin{split}
  &\mathbb{E}_{t_1}^\gamma[|X_t^{x,u}-X_t^{x',v^{0,n}}|^2]
  \leq|x-x'|^2+2\mathbb{E}_{t_1}^\gamma\Big[\int_0^t\Big(\langle X_s^{x,u}-X_s^{x',v^{0,n}},b(X_s^{x,u},u_s)-b(X_s^{x',v^{0,n}},v_s^{0,n})\rangle\\
   &+\frac{1}{2}|\sigma(X_s^{x,u},u_s)-\sigma(X_s^{x',v^{0,n}},v_s^{0,n})|^2+K_z|\sigma(X_s^{x,u},u_s)-\sigma(X_s^{x',v^{0,0}},v_s^{0,0})||X_s^{x,u}-X_s^{x',v^{0,0}}|\Big) ds \Big]\\
   =&|x-x'|^2+2\mathbb{E}_{t_1}^\gamma[\int_0^tg(X_s^{x,u},X_s^{x',v^{0,n}},u_s,v_s^{0,n})ds]\\
   \leq& |x-x'|^2+2\mathbb{E}_{t_1}^\gamma[ct1_{\{t>\tau_n\}}+t\delta 1_{\{t\leq\tau_n\}}]\\
   \leq& |x-x'|^2+\frac{ct_1^{p+1}}{\delta^{2p}n^{p-1}}e^{K_z^2t_1}+ct_1\delta,\ t\in[0,t_1].
   \end{split}
\end{equation}
For this we have used the definition of $\tau_n$ and the boundedness of $g$ over $\overline{\theta}\times\overline{\theta}\times U\times U$.
Consequently,
\begin{equation}\label{r90}
\begin{split}
  &\lambda\int_0^{t_1}e^{-\lambda s}\mathbb{E}[L_s^\gamma|\psi(X_s^{x,u},\overline{Z}_s^{\lambda,x,u},u_s)-\psi(X_s^{x',v^{0,n}},\overline{Z}_s^{\lambda,x,u},v_s^{0,n})|]ds\\
  =&\lambda\int_0^{t_1}e^{-\lambda s}\mathbb{E}_{t_1}^\gamma[\widetilde{\psi}(X_s^{x,u},X_s^{x',v^{0,n}},\overline{Z}_s^{\lambda,x,u},u_s,v_s^{0,n})]ds+\lambda\int_0^{t_1}e^{-\lambda s}\mathbb{E}_{t_1}^\gamma[\overline{c}_0|X_s^{x,u}-X_s^{x',v^{0,n}}|]ds\\
  \leq&\lambda\int_0^{t_1}e^{-\lambda s}\mathbb{E}_{t_1}^\gamma[\delta1_{\{\tau_n\geq s\}}+\widetilde{\psi}(X_s^{x,u},X_s^{x',v^{0,n}},\overline{Z}_s^{\lambda,x,u},u_s,v_s^{0,n})1_{\{\tau_n<s\}}]ds\\
  &+\overline{c}_0\lambda\int_0^{t_1}e^{-\lambda s}(| x-x'|+\frac{ct_1^{\frac{p+1}{2}}}{\delta^pn^{\frac{p-1}{2}}}e^{\frac{1}{2}K_z^2t_1}+(ct_1\delta)^{\frac{1}{2}})ds.
\end{split}
\end{equation}
We remark that
\begin{equation*}
  |\psi(X_s^{x,u},\overline{Z}_s^{\lambda,x,u},u_s)|\leq|\psi(X_s^{x,u},0,u_s)|+K_z|\overline{Z}_s^{\lambda,x,u}|\leq M+K_z| \overline{Z}_s^{\lambda,x,u}|,
\end{equation*}
and as the same estimate holds true for $\psi(X_s^{x',v^{0,n}},\overline{Z}_s^{\lambda,x,u},v^{0,n}_s)$, we have
\begin{equation*}
  \widetilde{\psi}(X_s^{x,u},X_s^{x',v^{0,n}},\overline{Z}_s^{\lambda,x,u},u_s,v_s^{0,n})\leq2M+2K_z|\overline{Z}_s^{\lambda,x,u}|,\ s\in[0,t_1].
\end{equation*}
Thus,
\begin{equation*}
\begin{array}{lll}
  &\displaystyle\lambda\int_0^{t_1}e^{-\lambda s}\mathbb{E}_{t_1}^\gamma[\widetilde{\psi}(X_s^{x,u},X_s^{x',v^{0,n}},\overline{Z}_s^{\lambda,x,u},u_s,v_s^{0,n})1_{\{\tau_n<s\}}]ds\\
  &\displaystyle\leq2M\mathbb{E}_{t_1}^\gamma[1_{\{\tau_n<t_1\}}]\lambda\int_0^{t_1}e^{-\lambda s}ds+2K_z\lambda\int_0^{t_1}e^{-\lambda s}\mathbb{E}_{t_1}^\gamma[|\overline{Z}_s^{\lambda,x,u}|1_{\{\tau_n<s\}}]ds,
 \end{array}
\end{equation*}
where
\begin{equation*}
\lambda\int_0^{t_1}e^{-\lambda s}\mathbb{E}_{t_1}^\gamma[|\overline{Z}_s^{\lambda,x,u}|1_{\{\tau_n<s\}}]ds\leq \lambda(\mathbb{E}[\int_0^{t_1}|e^{-\lambda s}\overline{Z}_s^{\lambda,x,u}|^2ds])^{\frac{1}{2}}\cdot \sqrt{t_1}\cdot(\mathbb{E}[(L_{t_1}^\gamma)^4])^{\frac{1}{4}}(P\{\tau_n<t_1\})^{\frac{1}{4}}.
\end{equation*}
Recall that
\begin{equation*}
\mathbb{E}[\int_0^{t_1}|e^{-\lambda s}\overline{Z}_s^{\lambda,x,u}|^2ds]\leq 2(\frac{M}{\lambda})^2(2+\frac{K_z^2}{\lambda}),
\end{equation*}
and observe that
\begin{equation*}
 (\mathbb{E}[(L_{t_1}^\gamma)^4])^{\frac{1}{4}}\leq e^{2K_z^2t_1},
\end{equation*}
and
\begin{equation*}
 \mathbb{P}\{\tau_n<t_1\}\leq\sum\limits_{i=1}^n\mathbb{P}\{\tau_i-\tau_{i-1}<\frac{t_1}{n}\}\leq\frac{C_p^2t_1^{2p}}{\delta^{4p}n^{2p-1}}.
\end{equation*}
Hence,
\begin{equation*}
  \lambda\int_0^{t_1}e^{-\lambda s}\mathbb{E}_{t_1}^\gamma[|\overline{Z}_s^{\lambda,x,u}|1_{\{\tau_n<s\}}]ds\leq C_{M,\lambda}e^{2K_z^2t_1}\frac{C_p^{\frac{1}{2}}t_1^{\frac{p+1}{2}}}{\delta^{p}n^{\frac{p-1}{2}}},
\end{equation*}
and supposing without loss of generality that $\delta\in(0,1),\ K_z\geq1$ and $t_1(=t_1-t_0)\leq1$, we get from (\ref{r90}), (\ref{112}) and the above estimates,
\begin{equation*}
\begin{split}
  &\lambda\int_0^{t_1}e^{-\lambda s}\mathbb{E}[L_s^\gamma| \psi(X_s^{x,u},\overline{Z}_s^{\lambda,x,u},u_s)-\psi(X_s^{x',v^{0,n}},\overline{Z}_s^{\lambda,x,u},v_s^{0,n})|]ds\\
  \leq& \lambda\int_0^{t_1}e^{-\lambda s}ds\cdot(\overline{c}_0| x-x'|+c\delta^{\frac{1}{2}}+c_p\frac{t_1^{\frac{p}{2}}}{\delta^{p}n^{\frac{p-1}{2}}}e^{2K_z^2t_1}).
  \end{split}
\end{equation*}
Recall that $\delta\in(0,1)$ is arbitrary. Thus, choosing $\delta>0$ sufficiently small and $n$ large enough, we have for $v^0:=v^{0,n}\in\mathcal{U}$
\begin{equation}\label{t2}
\begin{array}{lll}
&{\rm (i)}\ (\mathbb{E}[L_t^\gamma|X_t^{x,u}-X_t^{x',v^{0}}|^2])^{\frac{1}{2}}\leq | x-x'|+\varepsilon\frac{t_1}{(\overline{c}_0(T+2))^M},\ t\in[0,t_1];\\
 &{\rm (ii)}\ \displaystyle\lambda \int_0^{t_1}e^{-\lambda s}\mathbb{E}[L_s^\gamma|\psi(X_s^{x,u},\overline{Z}_s^{\lambda,x,u},u_s)-\psi(X_s^{x',v^{0}},\overline{Z}_s^{\lambda,x,u},v_s^{0})|]ds\\
&\displaystyle\ \ \ \ \ \ \ \leq\lambda\int_0^{t_1}e^{-\lambda s}ds\cdot\overline{c}_0(|x-x'|+\frac{\varepsilon}{(\overline{c}_0(T+2))^M}),
\end{array}
\end{equation}
for all $\gamma\in L_{\mathbb{F}}^\infty(0,\infty;\mathbb{R}^d)$ with $|\gamma_s|\leq K_z$, dsd$\mathbb{P}$-a.e.\\
\noindent\textbf{Step 2.} We consider the interval $[0,t_2]$: Starting now from $(X_{t_1}^{x,u},X_{t_1}^{x',v^0})$ at time $t_1$, and with $u=u_{t_1}$ on $[t_1,t_2]$, we construct $v^1$. We begin with putting\\ $v_t^{1,0}:=v_t^01_{[0,t_1)}(t)+\widehat{v}(X_{t_1}^{x,u},X_{t_1}^{x',v^0},u_t)1_{[t_1,t_2]}(t)
=v_t^01_{[0,t_1)}(t)+\widehat{v}(X_{t_1}^{x,u},X_{t_1}^{x',v^0},u_{t_1})1_{[t_1,t_2]}(t), t\in[0,t_2]$. Similar to Step 1, we construct a sequence of control processes $(v^{1,n})_{n\geq0}$. Letting $n$ be large enough, there exists $v^1:=v^{1,n}$ such that, for all $\gamma\in L_{\mathbb{F}}^\infty(0,\infty;\mathbb{R}^d)$ with $|\gamma_s|\leq K_z$, dsd$\mathbb{P}$-a.e.,
\begin{equation}\label{t3}
(\mathbb{E}_{t_2}^\gamma[|X_t^{x,u}-X_t^{x',v^1}|^2\big|\mathcal{F}_{t_1}])^{\frac{1}{2}}=(\mathbb{E}[\frac{L_{t_2}^\gamma}{L_{t_1}^\gamma}\mid X_{t}^{x,u}-X_{t}^{x',v^1}|^2\big|\mathcal{F}_{t_1}])^{\frac{1}{2}}\leq|X_{t_1}^{x,u}-X_{t_1}^{x',v^0}|+\varepsilon\frac{t_2-t_1}{(\overline{c}_0(T+2))^M},
\end{equation}
for all $t\in[t_1,t_2]$, and
\begin{equation*}
\begin{split}
 & \lambda\int_{t_1}^{t_2} e^{-\lambda s}\mathbb{E}_s^\gamma[|\psi(X_s^{x,u},\overline{Z}_s^{\lambda,x,u},u_s)-\psi(X_s^{x',v^1},\overline{Z}_s^{\lambda,x,u},v_s^1)|\big|\mathcal{F}_{t_1}]ds\\
\leq&\lambda\int_{t_1}^{t_2} e^{-\lambda s}ds\cdot\overline{c}_0(| X_{t_1}^{x,u}-X_{t_1}^{x',v^0}|+\frac{\varepsilon}{(\overline{c}_0(T+2))^M}).
\end{split}
\end{equation*}
Then, from (\ref{t2}) and (\ref{t3}),
\begin{equation*}
\begin{split}
  &(\mathbb{E}_{t_2}^\gamma[|X_t^{x,u}-X_t^{x',v^1}|^2])^{\frac{1}{2}}=(\mathbb{E}_{t_1}^\gamma[((\mathbb{E}_{t_2}^\gamma[|X_t^{x,u}-X_t^{x',v^1}|^2\big|\mathcal{F}_{t_1}])^\frac{1}{2})^2])^\frac{1}{2}\\
  & \leq (\mathbb{E}_{t_1}^\gamma[(|X_{t_1}^{x,u}-X_{t_1}^{x',v^0}|+\varepsilon\frac{t_2-t_1}{(\overline{c}_0(T+2))^M})^2])^\frac{1}{2}\\
  & \leq (\mathbb{E}_{t_1}^\gamma[|X_{t_1}^{x,u}-X_{t_1}^{x',v^0}|^2])^\frac{1}{2}+\varepsilon\frac{t_2-t_1}{(\overline{c}_0(T+2))^M}\\
  & \leq (|x-x'|+\varepsilon\frac{t_1}{(\overline{c}_0(T+2))^M})+\varepsilon\frac{t_2-t_1}{(\overline{c}_0(T+2))^M}.
\end{split}
\end{equation*}
This combined once more with the result (\ref{t2}) of Step 1 yields
\begin{equation*}
  \sup\limits_{|\gamma|\leq K_z}(\mathbb{E}_t^\gamma[|X_t^{x,u}-X_t^{x',v^1}|^2)^{\frac{1}{2}}\leq |x-x'|+\varepsilon\frac{t_2}{(\overline{c}_0(T+2))^M},\ t\in[0,t_2].
\end{equation*}
On the other hand, arguing similarly with using (\ref{t2}) and (\ref{t3}), we get
\begin{equation*}
\begin{split}
& \lambda\int_{t_1}^{t_2} e^{-\lambda s}\mathbb{E}_s^\gamma[|\psi(X_s^{x,u},\overline{Z}_s^{\lambda,x,u},u_s)-\psi(X_s^{x',v^1},\overline{Z}_s^{\lambda,x,u},v_s^1)|]ds\\
\leq&\lambda\int_{t_1}^{t_2} e^{-\lambda s}ds\cdot\overline{c}_0((\mathbb{E}_{t_1}^\gamma[| X_{t_1}^{x,u}-X_{t_1}^{x',v^1}|^2])^{\frac{1}{2}}+\frac{\varepsilon}{(\overline{c}_0(T+2))^M})\\
\leq& \lambda\int_{t_1}^{t_2} e^{-\lambda s}ds(\overline{c}_0| x-x'|+\frac{\overline{c}_0(t_1+1)\varepsilon}{(\overline{c}_0(T+2))^{M}}),
\end{split}
\end{equation*}
which combined with the corresponding estimate (\ref{t2}) of Step 1 yields
\begin{equation*}
\begin{split}
 & \lambda\int_0^{t_2}e^{-\lambda s}\mathbb{E}_s^\gamma[|\psi(X_s^{x,u},\overline{Z}_s^{\lambda,x,u},u_s)-\psi(X_s^{x',v^1},\overline{Z}_s^{\lambda,x,u},v_s^1)|]ds\\
 \leq&\lambda\int_0^{t_2}e^{-\lambda s}ds(\overline{c}_0|x-x'|+\frac{\overline{c}_0(t_1+2)\varepsilon}{(\overline{c}_0(T+2))^M})\\
 \leq&\lambda\int_0^{t_2}e^{-\lambda s}ds(\overline{c}_0|x-x'|+\frac{\varepsilon}{(\overline{c}_0(T+2))^{M-1}}),
\end{split}
\end{equation*}
for all $\gamma\in L_{\mathbb{F}}^\infty(0,\infty;\mathbb{R}^d)$ with $|\gamma_s|\leq K_z$, dsd$\mathbb{P}$-a.e.

Similarly, we make our construction on $[t_2,t_3]$, $[t_3,t_4]$, $\cdots$, $[t_{M-1},t_M]$, to finally get a process $v^{M-1}$ defined on $[0,T]$, such that
\begin{equation*}
\left\{
\begin{split}
  &\mbox{(i)}\ \sup\limits_{|\gamma|\leq K_z}(\mathbb{E}_t^\gamma[|X_t^{x,u}-X_t^{x',v^{M-1}}|^2])^{\frac{1}{2}}\leq |x-x'|+\varepsilon\frac{T}{(\overline{c}_0(T+2))^M}\leq|x-x'|+\frac{\varepsilon}{2},\ t\in[0,T];\\
  &\mbox{(ii)}\displaystyle\  \lambda\int_0^{T}e^{-\lambda s}\mathbb{E}_s^\gamma[|\psi(X_s^{x,u},\overline{Z}_s^{\lambda,x,u},u_s)-\psi(X_s^{x',v^{M-1}},\overline{Z}_s^{\lambda,x,u},v_s^{M-1})|]ds\\
  &\ \ \ \ \ \ \leq\displaystyle \lambda\int_0^{T}e^{-\lambda s}ds(\overline{c}_0|x-x'|+\frac{\varepsilon}{2}),
\end{split}
\right.
\end{equation*}
for all $\gamma\in L_{\mathbb{F}}^\infty(0,\infty;\mathbb{R}^d)$ with $|\gamma_s|\leq K_z$, dsd$\mathbb{P}$-a.e. Here we have supposed without loss of generality that $\overline{c}_0\geq1$.
Let now $T=1, \rho=\rho_1=\frac{1}{2}, \widetilde{v}^1:=v^{M-1}$; we can make the same construction on $[1,2]$, starting with $(X_1^{x,u},X_1^{x',v^{M-1}},u_1)$, but now for $\frac{\varepsilon}{4}.$ Thus, we get $v^2$ on $[1,2]$, $\widetilde{v}^2:=\widetilde{v}^11_{[0,1]}+v^21_{[1,2]}$. Similarly, by iteration, for $\frac{\varepsilon}{2^{j+1}}$, we make our construction on $[j,j+1], j\geq2$. Then we get the construction of $v\in\mathcal{U}$ such that
\begin{equation*}
\begin{array}{lll}
 &\displaystyle {\rm (i)}\ (\mathbb{E}_t^\gamma[|X_t^{x,u}-X_t^{x',v}|^2])^{\frac{1}{2}}\leq |x-x'|+\varepsilon(\sum_{j=1}^\infty\frac{1}{2^j})=|x-x'|+\varepsilon,\ t\geq0,\\
  &\displaystyle {\rm (ii)}\ \lambda\int_0^\infty e^{-\lambda s}\mathbb{E}_s^\gamma[|\psi(X_s^{x,u},\overline{Z}_s^{\lambda,x,u},u_s)-\psi(X_s^{x',v},\overline{Z}_s^{\lambda,x,u},v_s)|]ds\\
 &\displaystyle \ \ \ \  \leq\lambda\int_0^\infty e^{-\lambda s}ds(\overline{c}_0|x-x'|+\varepsilon)=\overline{c}_0|x-x'|+\varepsilon,
\end{array}
\end{equation*}
for all $\gamma\in L_{\mathbb{F}}^\infty(0,\infty;\mathbb{R}^d)$ with $|\gamma_s|\leq K_z$, dsd$\mathbb{P}$-a.e.
\end{proof}
\begin{lemma}\label{lem:2.6}
We suppose that (\ref{r1}), (\ref{r48}) and (\ref{r49}) hold. Then the family of functions $\{\lambda V_\lambda\}_\lambda$ is equicontinuous and equibounded on $\overline{\theta}$. Indeed, for the constants $\overline{c}_0>0$, $M>0$ defined in (\ref{r48}), it holds that, for all $\lambda>0$, and for all $ x,x'\in\overline{\theta},$
\begin{equation*}
\left\{
\begin{array}{ll}
{\rm{(i)}}\ |\lambda V_\lambda(x)-\lambda V_\lambda(x')|\leq \overline{c}_0|x-x'|, \\
{\rm{(ii)}}\ |\lambda V_\lambda(x)|\leq M.
\end{array}
\right.
\end{equation*}

\end{lemma}
\begin{proof}
From Proposition \ref{th:2.4} we know that for all $t\geq0,\ \lambda>0$, $|\overline{Y}_t^{\lambda,x,u}|\leq\frac{M}{\lambda}$. Thus we have
\begin{equation*}
  |\lambda V_\lambda(x)|\leq\lambda\sup\limits_{u\in\mathcal{U}}|Y_0^{\lambda,x,u}|\leq M.
\end{equation*}
It remains to prove (i). Let $\lambda>0,\ x,\ x'\in\mathbb{R}^N$. For any $\varepsilon>0$, let  $u\in\mathcal{U}$ be such that
\begin{equation}\label{lg1}
V_\lambda(x)\geq \overline{Y}_0^{\lambda,x,u}-\frac{\varepsilon}{\lambda}.
\end{equation}
Then, we have from Proposition \ref{p:2.1} that, there is $v^\varepsilon\in\mathcal{U}$ such that (\ref{r50}) holds true.

Let us define $Y_s^\varepsilon=\overline{Y}_s^{\lambda,x,u}-\overline{Y}_s^{\lambda,x',v^\varepsilon}, Z_s^\varepsilon=\overline{Z}_s^{\lambda,x,u}-\overline{Z}_s^{\lambda,x',v^\varepsilon},$ and
\begin{equation*}
\gamma_s^\varepsilon=\frac{\psi(X_s^{x',v^\varepsilon},\overline{Z}_s^{\lambda,x,u},v_s^\varepsilon)-
\psi(X_s^{x',v^\varepsilon},\overline{Z}_s^{\lambda,x',v^\varepsilon},v_s^\varepsilon)}
{|\overline{Z}_s^{\lambda,x,u}-\overline{Z}_s^{\lambda,x',v^\varepsilon}|^2}
\cdot(\overline{Z}_s^{\lambda,x,u}-\overline{Z}_s^{\lambda,x',v^\varepsilon})^*,\ \mbox{if}\  \overline{Z}_s^{\lambda,x,u}\neq\overline{Z}_s^{\lambda,x',v^\varepsilon};
\end{equation*}
\noindent otherwise, $\gamma_s^\varepsilon=0,\  s\geq0$, where $(\overline{Y}^{\lambda,x,u},\overline{Z}^{\lambda,x,u})$ and $(\overline{Y}^{\lambda,x',v^\varepsilon},\overline{Z}^{\lambda,x',v^\varepsilon})$ are the solutions of BSDE (\ref{r40}) with the driving coefficient $\psi(X^{x,u},\cdot,u)$ and $\psi(X^{x',v^\varepsilon},\cdot,v^\varepsilon)$, respectively. We note that from (\ref{r48}) it follows that $|\gamma_s^\varepsilon|\leq K_z$. Putting
\begin{equation*}
  L_s^\varepsilon=\exp\{\int_0^s\gamma^\varepsilon_rdW_r-\frac{1}{2}\int_0^s|\gamma^\varepsilon_r|^2dr\},\ s\geq0,
\end{equation*}
we define probability measures $\mathbb{P}_s^\varepsilon$ on $(\Omega,\mathcal{F})$ by setting
\begin{equation*}
  \frac{d\mathbb{P}_s^\varepsilon}{d\mathbb{P}}=\exp\{\int_0^s\gamma^\varepsilon_rdW_r-\frac{1}{2}\int_0^s|\gamma^\varepsilon_r|^2dr\},\ s\geq0.
\end{equation*}
Then, it follows from Girsanov's theorem that
\begin{equation*}
  W^\varepsilon_t=W_t-\int_0^t\gamma^\varepsilon_rdr,\ t\in[0,s],
\end{equation*}
is an $(\mathbb{F,\mathbb{P}}_s^\varepsilon)$-Brownian motion. Then, for all $0\leq t\leq T<\infty$,
\begin{equation*}
\begin{split}
  Y_t^\varepsilon=&Y_T^\varepsilon-\lambda\int_t^TY_s^\varepsilon ds+\int_t^T(\psi(X_s^{x,u},\overline{Z}_s^{\lambda,x,u},u_s)-\psi(X_s^{x',v^\varepsilon},\overline{Z}_s^{\lambda,x',v^\varepsilon},v_s^\varepsilon))ds
  -\int_t^TZ_s^\varepsilon dW_s\\
  =& Y_T^\varepsilon-\lambda\int_t^TY_s^\varepsilon ds+\int_t^T(\psi(X_s^{x,u},\overline{Z}_s^{\lambda,x,u},u_s)-\psi(X_s^{x',v^\varepsilon},\overline{Z}_s^{\lambda,x,u},v_s^\varepsilon))ds
  -\int_t^TZ_s^\varepsilon dW^\varepsilon_s.
  \end{split}
\end{equation*}
By applying It\^{o}'s formula to $e^{-\lambda t}Y_t^\varepsilon$, and taking the conditional expectation $\mathbb{E}^\varepsilon_T[\cdot\big|\mathcal{F}_t]$ with respect to $\mathbb{P}^\varepsilon_T$, we obtain
\begin{equation*}
\begin{split}
  Y_t^\varepsilon=&\mathbb{E}_T^\varepsilon[\int_t^Te^{-\lambda (s-t)}\big(\psi(X_s^{x,u},\overline{Z}_s^{\lambda,x,u},u_s)-\psi(X_s^{x',v^\varepsilon},\overline{Z}_s^{\lambda,x,u},v_s^\varepsilon)\big)ds\big|\mathcal{F}_t]
  +\mathbb{E}_T^\varepsilon[e^{-\lambda(T-t)}Y_T^\varepsilon\big|\mathcal{F}_t].
\end{split}
\end{equation*}
Let $t=0$. Since $|Y_T^\varepsilon|=|\overline{Y}_T^{\lambda,x,u}-\overline{Y}_T^{\lambda,x',v^\varepsilon}|\leq \frac{2M}{\lambda}$, it follows that
\begin{equation*}
\begin{split}
  |Y_0^\varepsilon|\leq& \frac{2M}{\lambda}e^{-\lambda T}+\int_0^Te^{-\lambda s}\mathbb{E}[L_s^\varepsilon\mid\psi(X_s^{x,u},\overline{Z}_s^{\lambda,x,u},u_s)-\psi(X_s^{x',v^\varepsilon},\overline{Z}_s^{\lambda,x,u},v_s^\varepsilon)\mid]ds\\
  \leq& \frac{2M}{\lambda}e^{-\lambda T}+\frac{\overline{c}_0}{\lambda}(|x-x'|+\varepsilon),\ T\geq0.
\end{split}
\end{equation*}
Here we have used the fact that due to the choice of $v^\varepsilon$ (\ref{r50}) is satisfied.
Now letting $T$ tend to infinity we get
\begin{equation}\label{lg2}
  |\overline{Y}_0^{\lambda,x,u}-\overline{Y}_0^{\lambda,x',v^\varepsilon}|\leq \frac{\overline{c}_0}{\lambda}(|x-x'|+\varepsilon).
\end{equation}
Finally, from the arbitrariness of $u\in\mathcal{U}$ and $\varepsilon>0$ it follows that
\begin{equation*}
  |\lambda V_\lambda(x)-\lambda V_\lambda(x')|\leq \overline{c}_0|x-x'|.
\end{equation*}
\noindent Indeed, from (\ref{lg1}) and (\ref{lg2}) we have
\begin{equation*}
  \lambda V_\lambda(x)-\lambda V_\lambda(x')\geq\lambda(\overline{Y}_0^{\lambda,x,u}-\overline{Y}_0^{\lambda,x',v^\varepsilon})-\varepsilon\geq-\overline{c}_0|x-x'|-(\overline{c}_0+1)\varepsilon,
\end{equation*}
and letting $\varepsilon\downarrow0$ yields $\lambda V_\lambda(x)-\lambda V_\lambda(x')\geq-\overline{c}_0|x-x'|$. The symmetry of the argument in $x$ and $x'$ gives the inverse inequality.
The proof is complete.
\end{proof}

\section{ {\protect \large Hamilton-Jacobi-Bellman equations}}
Before we study in the next section the HJB equations associated with the stochastic control problem (\ref{r94}), let us begin a more general discussion in this section, where we consider a Hamiltonian $H:\mathbb{R}^N\times\mathbb{R}^N\times\mathcal{S}^N\to\mathbb{R}$ not necessarily related with a stochastic control problem. By $\mathcal{S}^N$ we denote the set of symmetric $N\times N$ matrices.

Let $H:\mathbb{R}^N\times\mathbb{R}^N\times\mathcal{S}^N\to\mathbb{R}$ be a uniformly continuous function satisfying the monotonicity assumption:\\
\noindent\textbf{($A_H)$} (i) $H(x,p,A)\leq H(x,p,B),\ \text{for\ all}\ (x,p)\in\mathbb{R}^N\times\mathbb{R}^N,\ A,\ B\in\mathcal{S}^N \mbox{with}\ B\leq A$ (i.e., $A-B$ is positive semidefinite).\\

We consider the PDE
\begin{equation}\label{r3.1}
\lambda V(x)+H(x,DV(x),D^2V(x))=0,\ x\in\overline{\theta}.
\end{equation}
Let $V:\overline{\theta}\to\mathbb{R}$ be a bounded measurable function. We define
\begin{equation*}
  V^*(x)=\varlimsup\limits_{y\to x}V(y),\  V_*(x)=\varliminf\limits_{y\to x}V(y),\ x\in\overline{\theta}.
\end{equation*}
Then, $V^*:\overline{\theta}\to\mathbb{R}$ is upper semicontinuous (we write $V^*\in \mbox{USC}(\overline{\theta})$) and $V_*:\overline{\theta}\to\mathbb{R}$ is lower semicontinuous ($V_*\in \mbox{LSC}(\overline{\theta})$).
\begin{definition}\label{def3.1}
$V$ is a constrained viscosity solution of (\ref{r3.1}), if it solves
\begin{equation*}
  \lambda V(x)+H(x,DV(x),D^2V(x))=0,\ x\in\theta,
\end{equation*}
\begin{equation*}
  \lambda V(x)+H(x,DV(x),D^2V(x))\geq0,\ x\in\partial\theta,
\end{equation*}
in viscosity sense, i.e., if\\
\rm{i)} \emph{$V$ is a viscosity subsolution of (\ref{r3.1}) on $\theta$, and}\\
\rm{ii)} \emph{$V$ is a viscosity supersolution of (\ref{r3.1}) on $\overline{\theta}$}.
\end{definition}
\begin{remark}
Recall that\\
\rm{i)} \emph{$V$ is a viscosity subsolution of (\ref{r3.1}) on $\theta$, if for all $x\in\theta$ and all $\varphi\in C^2(\mathbb{R}^N)$ such that $V^*-\varphi$ achieves a local maximum on $\theta$ at $x$, it holds}
\begin{equation*}
  \lambda V^*(x)+H(x,D\varphi(x),D^2\varphi(x))\leq0;
\end{equation*}
\rm{ii)} \emph{$V$ is a viscosity supersolution of (\ref{r3.1}) on $\overline{\theta}$, if for all $x\in\overline{\theta}$ and all $\varphi\in C^2(\mathbb{R}^N)$ such that
$V_*-\varphi$ achieves a local minimum on $\overline{\theta}$ at $x$, it holds}
\begin{equation*}
  \lambda V_*(x)+H(x,D\varphi(x),D^2\varphi(x))\geq0.
\end{equation*}
\emph{The reader can refer to Crandall, Ishii, Lions \cite{ISHII 1992}.}
\end{remark}
Existence and comparison results for the viscosity solution of (\ref{r3.1}) have been established (see Theorem 2.1, 2.2 and 3.1 in Katsoulakis \cite{Katsoulakis 1994}) under the additional assumptions:\\
\noindent\textbf{($A_\theta)$}  There exists a bounded, uniformly continuous function $m:\overline{\theta}\rightarrow \mathbb{R}^N$ with $|m|\leq 1$ and a\\
\indent\ constant $r>0$ such that\\
\centerline{$B(x+sm(x),rs)\subset \theta , \ \text{for\ all}\ x\in \overline{\theta},\ s\in(0,r].$}
\indent\ $B(x,s)\subset\mathbb{R}^N$ denotes the open ball with center at $x$ and radius $s$.\\
\noindent\textbf{($A_H)$} (ii) There is a continuity modulus $\rho: \mathbb{R}_+\to\mathbb{R}_+$, $\rho(0+)=0$, such that,\\
\indent\qquad\ $|H(x,p,A)-H(x,q,B)|\leq\rho(|p-q|+|A-B|),\ x,\ p,\ q\in\mathbb{R}^N, A,\ B\in\mathcal{S}^N,\ \mbox{and}$\\
\indent \qquad\ $|H(y,p,B)-H(x,p,A)|\leq\rho(\frac{1}{\varepsilon}|x-y|^2+|x-y|(|p|+1)),\ \mbox{for\ all}\ x,\ y,\ p\in\mathbb{R}^N,\ \varepsilon>0,$ \indent \qquad\ $A,\ B\in\mathcal{S}^N,$ such that
$$-\frac{3}{\varepsilon}\begin{pmatrix}
I&0\\
0&I
\end{pmatrix}\leq \begin{pmatrix}
A&0\\
0&B
\end{pmatrix}\leq \frac{3}{\varepsilon}\begin{pmatrix}
I&-I\\
-I&I
\end{pmatrix},$$
\indent \qquad\ where $I\in\mathbb{R}^{N\times N}$ denotes the unit matrix in $\mathbb{R}^{N\times N}$.\\

Under the above assumptions Katsoulakis \cite{Katsoulakis 1994} has shown the following theorems:
\begin{theorem}\label{th1}(Comparison principle; Theorem 2.2 in \cite{Katsoulakis 1994})
Let $u\in USC(\overline{\theta})$ be a subsolution of (\ref{r3.1}) on $\theta$ and $v\in LSC(\overline{\theta})$ a supersolution of (\ref{r3.1}) on $\overline{\theta}$. Then $u\leq v$ on $\overline{\theta}$.
\end{theorem}
\begin{remark}
In Theorem 2.2 in \cite{Katsoulakis 1994} the condition on $u$ is slightly weaker formulated: $u\in USC(\theta)$ is nontangential upper semicontinuous on $\partial\theta$.
\end{remark}
\begin{theorem}\label{th2}(Existence; Theorem 3.1 in \cite{Katsoulakis 1994})
If, in addition to the above assumptions, there is a bounded supersolution $\widetilde{v}\in LSC(\overline{\theta})$ of (\ref{r3.1}), then (\ref{r3.1}) has a constrained viscosity solution $v\in LSC(\overline{\theta})$; it is given by the smallest supersolution of (\ref{r3.1}) in $LSC(\overline{\theta})$.
\end{theorem}
\begin{remark}
Let us suppose that $H$ satisfies the radial monotonicity assumption, which is introduced in Theorem \ref{th:3.3}. Then
\begin{equation}\label{r3.2}
H(x,p,A)\geq H(x,0,0),\ (x,p,A)\in\overline{\theta}\times\mathbb{R}^N\times\mathcal{S}^N\ (\mbox{see\ Lemma \ref{l:3.4}}).
\end{equation}
Furthermore, let $K\in\mathbb{R}$ be such that $K\geq-H(x,0,0),\ x\in\overline{\theta}$. Then one checks easily that $\widetilde{v}(x)=\frac{K}{\lambda},\  x\in\overline{\theta}$, is a viscosity supersolution of (\ref{r3.1}) on $\overline{\theta}$.\\
Indeed, if $x\in\overline{\theta}$ and $\varphi\in C^2(\mathbb{R}^N)$ such that $\widetilde{v}-\varphi\geq\widetilde{v}(x)-\varphi(x)$ on $\overline{\theta}$, then clearly
\begin{equation*}
 \lambda \widetilde{v}(x)+ H(x,D\varphi(x),D^2\varphi(x))\geq 0,\  \ x\in\overline{\theta},
\end{equation*}
with the help of (\ref{r3.2}).
\end{remark}
Let us introduce the space $Lip_{M_0}(\overline{\theta})$ ($M_0>0$) as space of all Lipschitz functions $u: \overline{\theta}\to\mathbb{R}$ with $|u(x)|\leq M_0$, $|u(x)-u(y)|\leq M_0|x-y|, x,y\in\overline{\theta},$ and let us suppose that, for $M_0>0$ large enough,\\
(H) PDE (\ref{r3.1}) has a solution $V_\lambda$ such that $\lambda V_\lambda\in Lip_{M_0}(\overline{\theta})$, for all $\lambda>0$.
\begin{remark}
Under the assumption of the existence of a bounded supersolution on $\overline{\theta}$, we have from Theorem \ref{th2} the existence of a viscosity solution $V_\lambda\in LSC(\overline{\theta})$. With (H) we suppose that $\lambda V_\lambda\in Lip_{M_0}(\overline{\theta})$. The uniqueness of this solution $V_\lambda\in Lip_M(\overline{\theta})$ is guaranteed by the comparison priciple (Theorem \ref{th1}). Later, in the discussion of the case where the Hamiltonian $H$ is associated with the stochastic control problem (\ref{r94}), we will see that for such Hamiltonians PDE (\ref{r3.1}) has a unique constrained viscosity solution $\lambda V_\lambda\in Lip_{M_0}(\overline{\theta})$, for all $\lambda>0$.
\end{remark}
In what follows we work with the hypothesis (H).

Associated with our problem is the family of Hamiltonians
\begin{equation*}
H_\lambda(x,p,A):=\lambda H(x,\frac{1}{\lambda}p,\frac{1}{\lambda}A),\ (x,p,A)\in\overline{\theta}\times\mathbb{R}^N\times\mathcal{S}^N,\ \lambda>0,
\end{equation*}
where $H$ is supposed to satisfy ($A_H$).
\begin{theorem}\label{th:3.3} We suppose that, in addition to \textbf{($A_\theta$)}, \textbf{($A_H$)} and (H), the Hamiltonian $H$ satisfies the radial monotonicity condition:
\begin{equation}\label{r15}
H(x,l p,l A)\geq H(x,p,A),\ \text{for\ all\ real}\ l\geq 1,\ (x,p,A)\in\overline{\theta}\times \mathbb{R}^N\times \mathcal{S}^N.\tag{H5}
\end{equation}
For all $\lambda >0$, let $V_\lambda$ be the constrained viscosity solution of PDE (\ref{r3.1}) such that $\lambda V_\lambda\in \mbox{Lip}_M(\overline{\theta})$. Then \\
\rm{(i)} \emph{$\lambda\rightarrow\lambda V_\lambda(x)$ is nondecreasing, for every $x\in\overline{\theta}$;}\\
\rm{(ii)} \emph{The limit $\lim_{\lambda\rightarrow0^+}\lambda V_\lambda(x)$ exists, for every $x\in\overline{\theta}$;}\\
\rm{(iii)} \emph{The convergence in (ii) is uniform on $\overline{\theta}$.}
\end{theorem}
\begin{proof}
First, we know that for every $\lambda >0$, $w_\lambda(x):=\lambda V_\lambda(x)$ is a constrained viscosity solution of
\begin{equation}\label{r16}
\lambda w_\lambda(x)+H_\lambda(x,D w_\lambda(x),D^2w_\lambda(x))=0.
\end{equation}
For any $\lambda,\ \mu>0$, we have $\displaystyle \frac{\lambda}{\mu}H_\mu(x,\frac{\mu}{\lambda}p,\frac{\mu}{\lambda}A)=\frac{\lambda}{\mu}(\mu H(x,\frac{\mu}{\lambda}(\frac{1}{\mu}p),\frac{\mu}{\lambda}(\frac{1}{\mu}A)))=H_\lambda(x,p,A).
$\\
Using the radial monotonicity condition (\ref{r15}) we have, for any $\mu>\lambda>0$, in viscosity sense,
\begin{eqnarray*}
\begin{array}{lll}
  &\displaystyle\lambda w_\mu(x)+H_\lambda(x,D w_\mu(x),D^2w_\mu(x)) = \mu\cdot\frac{\lambda}{\mu} w_\mu(x)+\frac{\lambda}{\mu}H_\mu(x,\frac{\mu}{\lambda}D w_\mu(x),\frac{\mu}{\lambda}D^2w_\mu(x)) \\
 &\displaystyle   =\frac{\lambda}{\mu}[\mu w_\mu(x)+\mu H(x,\frac{\mu}{\lambda}(\frac{1}{\mu}D w_\mu(x)),\frac{\mu}{\lambda}(\frac{1}{\mu}D^2w_\mu(x)))] \\
 &\displaystyle   \geq\frac{\lambda}{\mu}[\mu w_\mu(x)+\mu H(x,\frac{1}{\mu}D w_\mu(x),\frac{1}{\mu}D^2w_\mu(x))] \\
 &\displaystyle   = \frac{\lambda}{\mu}(\mu w_\mu(x)+H_\mu(x,Dw_\mu(x),D^2w_\mu(x)))\geq0,\ x\in\overline{\theta}.
   \end{array}
\end{eqnarray*}
Therefore, $w_\mu\in Lip_{M_0}(\overline{\theta})$ is a viscosity supersolution to (\ref{r16}) on $\overline{\theta}$. From the comparison principle-Theorem \ref{th1}, $w_\mu\geq w_\lambda$ on $\overline{\theta}$. Statement (ii) follows from (i) and the boundedness of $\lambda V_\lambda$, $\lambda>0$, while thanks to the fact that $\lambda V_\lambda\in \mbox{Lip}_M(\overline{\theta}),\  \lambda>0$, the Arzel\`{a}-Ascoli Theorem yields (iii).
\end{proof}
\begin{lemma}\label{l:3.4}
Let $H(x,p,A)$ be convex in $(p,A)\in\mathbb{R}^N\times\mathcal{S}^N$. Then we have the following equivalence:\\
\rm{i)}\  \emph{The radial monotonicity (\ref{r15}) holds true for $H(x,\cdot,\cdot)$;}\\
\rm{ii)}\  $H(x,l'p,l'A)\geq H(x,lp,lA),\ 0\leq l\leq l',\ (p,A)\in\mathbb{R}^N\times\mathcal{S}^N$;\\
\rm{iii)}\  $H(x,p,A)\geq H(x,0,0),\ (p,A)\in\mathbb{R}^N\times\mathcal{S}^N$.
\end{lemma}
\begin{proof} Indeed, i) and ii) are obviously equivalent. Moreover, ii) implies iii) (take $l'=1$, and $l=0$). Thus, it remains only to show that iii) implies ii).
For this end, given any $(p,A)\in\mathbb{R}^N\times\mathcal{S}^N$, we consider the function $G(l):=H(x,lp,lA),\ l\geq0$. From the convexity of $H(x,\cdot,\cdot)$ it follows that of $G$, and iii) implies that $G(l)\geq G(0),\ l\geq0$. Consequently, for $l'\geq l\geq0$ and $\displaystyle k:=\frac{l}{l'}\in[0,1]$,
\begin{equation*}
  H(x,lp,lA)=G(l'k)=G(l'k+(1-k)0)\leq kG(l')+(1-k)G(0)\leq kG(l')+(1-k)G(l')=H(x,l'p,l'A).
\end{equation*}
\end{proof}
\begin{remark}
We suppose that $H$ is of the form (\ref{r20}) and $(-\psi)(x,\cdot,u)=\{z\mapsto(-\psi)(x,z,u)\}$ is convex, for all $(x,u)\in\overline{\theta}\times U$. Then $H(x,p,A)$ is convex in $(p,A)$, for all $x\in\overline{\theta}$.

Under the additional assumption of the existence of some $u_0\in U$ such that $b(x,u_0)=0,\ \sigma(x,u_0)=0$ and $\psi(x,0,u)\geq\psi(x,0,u_0)$, for all $u\in U$, we have
\begin{equation*}
\begin{split}
 &H(x,p,A)=\sup\limits_{u\in U}\{\langle-p,b(x,u)\rangle-\frac{1}{2}tr(\sigma\sigma^*(x,u)A)-\psi(x,p\sigma(x,u),u)\}\\
 &\geq\langle-p,b(x,u_0)\rangle-\frac{1}{2}tr(\sigma\sigma^*(x,u_0)A)-\psi(x,p\sigma(x,u_0),u_0)\\
 &=-\psi(x,0,u_0)=\sup\limits_{u\in U}\{-\psi(x,0,u)\}=H(x,0,0),\ \  (p,A)\in\mathbb{R}^N\times\mathcal{S}^N.
\end{split}
\end{equation*}
Then Lemma \ref{l:3.4} yields that $H$ satisfies the radial monotonicity condition.

However, without additional assumption for $H$ of the form (\ref{r20}), only with $(-\psi)(x,z,u)$ is convex in $z$, we don't, in general, have the radial monotonicity.

Indeed, for example, if, for some $\varepsilon>0$ and $x\in\overline{\theta}$, $\sigma\sigma^*(x,u)\geq\varepsilon1_{\mathbb{R}^N}, u\in U$, then
\begin{equation*}
\begin{split}
 & H(x,0,A)=\sup\limits_{u\in U}\{-\frac{1}{2}tr(\sigma\sigma^*(x,u)A)-\psi(x,0,u)\}\\
 &\leq-\frac{1}{2}\varepsilon tr(A)+\sup\limits_{u\in U}\{-\psi(x,0,u)\}=-\frac{1}{2}\varepsilon tr(A)+H(x,0,0)\\
 & < H(x,0,0), \ \mbox{for\ all}\ A\in\mathcal{S}^N \ \mbox{with}\ \ tr(A)>0.
\end{split}
\end{equation*}
\end{remark}
Under the assumptions of Theorem \ref{th:3.3} we let $w_0(x)=\lim\limits_{\lambda\rightarrow0^+}\lambda V_\lambda(x)$, $x\in\overline{\theta}$. Next we will characterize $w_0(x)$ under the condition of radial monotonicity of $H$ as maximal viscosity subsolution of the PDE
\begin{equation}\label{r3.4}
  W(x)+\overline{H}(x,DW(x),D^2W(x))=0,\ x\in\theta,
\end{equation}
where $\overline{H}(x,p,A):=\min\{M_0,\sup\limits_{l>0}H(x,lp,lA)\}$.
\begin{remark}
As $H:\overline{\theta}\times\mathbb{R}^N\times\mathcal{S}^N\to\mathbb{R}$ is continuous and $\sup\limits_{l\geq0}H(x,lp,lA)=\lim\limits_{l\to\infty}\uparrow H(x,lp,lA)$, the function $\overline{H}$ is lower semicontinuous. Recall that a function $W\in USC(\overline{\theta})$ is a viscosity subsolution of (\ref{r3.4}) on $\theta$, if for all $x\in\theta$, $\varphi\in C^2(\mathbb{R}^N)$ such that $W-\varphi\leq W(x)-\varphi(x)$ on $\theta$,
$$W(x)+\overline{H}(x,D\varphi(x),D^2\varphi(x))\leq0.$$
\end{remark}
\begin{theorem}\label{the:3.4}
We make the same assumptions as in Theorem \ref{th:3.3}. For all $\lambda>0$, let $V_\lambda$ be the unique constrained viscosity solution of the PDE
\begin{equation}\label{r99}
\lambda V(x)+H(x,DV(x),D^2V(x))=0,\ x\in \overline{\theta},
\end{equation}
such that $\lambda V_\lambda\in \mbox{Lip}_{M_0}(\overline{\theta})$, for some $M_0> 0$ large enough and independent of $\lambda$.
Then, $w_0(x):=\lim\limits_{\lambda\to0^+}\lambda V_\lambda(x)$, $x\in\overline{\theta}$, is the maximal viscosity subsolution of (\ref{r3.4}),
$$w_0(x)=\sup\{w(x):w\in \mbox{Lip}_{M_0}(\overline{\theta}), w+\overline{H}(x,Dw,D^2w)\leq 0 \ \mbox{on} \ \theta\ (\mbox{in\ viscosity\ sense})\},\ x\in\overline{\theta},$$
where $\overline{H}(x,p,A):=\min \Big\{M_0, \sup\limits_{l>0}H(x,l p,l A)\Big\}$.
\end{theorem}
\begin{proof}
We define the set
$$\mathcal{S}_{H,M_0}=\{w: w\in \mbox{Lip}_{M_0}(\overline{\theta}),\ w+\overline{{H}}(x,Dw,D^2w)\leq 0 \ \text{on} \ \theta\ (\mbox{in\ viscosity\ sense})\},$$ and we set $\bar{w}(x)=\sup\{w(x), w\in\mathcal{S}_{H,M_0}\}.$\\
\textbf{Step 1}. We show that $w_0$ is a viscosity subsolution of (\ref{r3.4}), which implies that $w_0\in\mathcal{S}_{H,M}$ and, thus $w_0\leq \bar{w}$.\\
\textbf{Step 1.1}. We first prove that $w_\lambda=\lambda V_\lambda(x)\in Lip_{M_0}(\overline{\theta})$ is also a constrained viscosity solution of the equation
\begin{equation}\label{rl7}
  w(x)+H^{M_0}(x,\frac{1}{\lambda}Dw(x),\frac{1}{\lambda}D^2w(x))=0,\ x\in\overline{\theta},
\end{equation}
where $H^{M_0}(x,p,A):=\min\{M_0, H (x,p,A)\},\ \text{for\ all}\ (x,p,A)\in \overline{\theta}\times \mathbb{R}^N\times \mathcal{S}^N.$

In fact, let $x\in\theta$ and $\phi\in C^2(\mathbb{R}^N)$ be such that $(w_\lambda-\phi)(x)=\max\{(w_\lambda-\phi)(\overline{x}),\ \overline{x}\in\overline{\theta}\}$.
Then, as $V_\lambda$ is a constrained viscosity solution of (\ref{r99}) and $w_\lambda=\lambda V_\lambda$, we have
$$w_\lambda(x)+H(x,\frac{1}{\lambda}D\phi(x),\frac{1}{\lambda}D^2\phi(x))\leq 0.$$
Furthermore, from $w_\lambda\in \mbox{Lip}_{M_0}(\overline{\theta})$ we get
$$H(x,\frac{1}{\lambda}D\phi(x),\frac{1}{\lambda}D^2\phi(x))\leq -w_\lambda(x)\leq M_0.$$
It follows that
$$w_\lambda(x)+H^{M_0}(x,\frac{1}{\lambda}D\phi(x),\frac{1}{\lambda}D^2\phi(x))=w_\lambda(x)+H(x,\frac{1}{\lambda}D\phi(x),\frac{1}{\lambda}D^2\phi(x))\leq 0,$$
i.e., $w_\lambda$ is a constrained subsolution of (\ref{rl7}) in $\overline{\theta}$.

Next we show that $w_\lambda$ is also a supersolution on $\overline{\theta}$. Let $x\in\overline{\theta}$ and $\varphi\in C^2(\mathbb{R}^N)$ be such that $(w_\lambda-\varphi)(x)=\min\{(w_\lambda-\varphi)(\overline{x}), \overline{x}\in\overline{\theta}\}$. Obviously, from (\ref{r99}) and the fact that $w_\lambda\in\mbox{Lip}_{M_0}(\overline{\theta})$ we have the following both inequalities,
\begin{equation*}
\left\{
\begin{array}{lll}
w_\lambda(x)+H(x,\frac{1}{\lambda}D\varphi(x),\frac{1}{\lambda}D^2\varphi(x))\geq 0,\\
w_\lambda(x)+M_0\geq 0.
\end{array}
\right.
\end{equation*}
Thus,
$w_\lambda(x)+H^{M_0}(x,\frac{1}{\lambda}D\varphi(x),\frac{1}{\lambda}D^2\varphi(x))\geq 0$, which implies (\ref{rl7}).\\
\textbf{Step 1.2}. Now we show $w_0\in\mathcal{S}_{H,M_0}$, i.e., $w_0\in\mbox{Lip}_{M_0}(\overline{\theta})$ and
\begin{equation}\label{r18}
w_0+\overline{H}(x,Dw_0,D^2w_0)\leq 0\ \mbox{in} \ \theta,
\end{equation}
in viscosity sense.

Indeed, let us fix $l>0$. Then (\ref{r15}) and (\ref{rl7}) yield, for any $0<\lambda\leq\frac{1}{l}$,
$$w_\lambda+H^{M_0}(x,l Dw_\lambda,l D^2w_\lambda)\leq w_\lambda+H^{M_0}(x,\frac{1}{\lambda}Dw_\lambda,\frac{1}{\lambda}D^2w_\lambda)=0 \ \ \mbox{in} \ \theta,$$
in viscosity sense. Recall that $w_\lambda\in\mbox{Lip}_{M_0}(\overline{\theta}), \lambda>0$, and that $w_0$ is the uniform limit of $w_\lambda$, as $\lambda\downarrow0$. Consequently, $w_0\in Lip_{M_0}(\overline{\theta})$. Moreover, by the result that the uniform limit of subsolutions is a subsolution again, we conclude that, in viscosity sense,
$$w_0+H^{M_0}(x,l Dw_0,l D^2w_0)\leq 0, \ l>0.$$
Finally, taking the supremum with respect to $l>0$ over the increasing left-hand side, it follows that (\ref{r18}) holds.

Consequently, $w_0\in\mathcal{S}_{H,M_0}$ and thus $w_0\leq \bar{w}$.\\
\textbf{Step 2}. Notice that also $\bar{w}\in\mbox{Lip}_{M_0}(\overline{\theta})$. Thus, in order to prove that $w_0\geq \bar{w}$, we need to check that
\begin{equation}\label{r19}
  \bar{w}+\overline{H}(x,D\bar{w},D^2\bar{w})\leq 0 \ \mbox{in} \ \theta.
\end{equation}
The above property of the upper envelope of a bounded family of subsolutions is well
known when $H$ is continuous and can be extended to $\overline{H}$. Let $x\in\theta$ and $\phi\in C^2(\mathbb{R}^N)$ be such that $(\bar{w}-\phi)(x)=\max\{(\bar{w}-\phi)(\overline{x}), \overline{x}\in\overline{\theta}\}$. By adding a constant to $\phi$ one can assume that $\bar{w}(x)=\phi(x)$. For $\varepsilon>0$ we put $\phi_\varepsilon(y)=\phi(y)+\varepsilon|x-y|^2$, $y\in\mathbb{R}^N$. Then, $$(\bar{w}-\phi_\varepsilon)(y)\leq -\varepsilon|y-x|^2,$$
for every $y$ in $\overline{\theta}$, and, hence, also in some closed ball $B_r(x)\subseteq\theta$. Thus, by the very definition of $\bar{w}$, there
exists a sequence$\{w^n\}_n\subseteq \mathcal{S}_{H,M_0}$ such that $w^n(x)\geq \bar{w}(x)-\frac{1}{n}$ for all $n\geq 1$.
Let $x_n^\varepsilon$ be a maximum point of $w^n-\phi_\varepsilon$ over $B_r(x)$. Then we have that
$$-\frac{1}{n}\leq (w^n-\phi_\varepsilon)(x)\leq (w^n-\phi_\varepsilon)(x_n^\varepsilon)\leq -\varepsilon|x_n^\varepsilon-x|^2.$$
Consequently, $x_n^\varepsilon\rightarrow x$ and $(w^n-\phi_\varepsilon)(x_n^\varepsilon)\rightarrow 0$, as $n\rightarrow\infty$. Therefore, $w^n(x_n^\varepsilon)\rightarrow \bar{w}(x)$, as $n\rightarrow\infty$. Moreover, for all $l >0$ and $n$ large enough, we have
\begin{equation}\label{th3.4.5}
w^n(x_n^\varepsilon)+H^{M_0}(x_n^\varepsilon,l D\phi(x_n^\varepsilon),l D^2\phi(x_n^\varepsilon))\leq w^n(x_n^\varepsilon)+\overline{H}(x_n^\varepsilon,D\phi(x_n^\varepsilon),D^2\phi(x_n^\varepsilon))\leq 0.
\end{equation}
On the other hand,
\begin{equation*}
  \overline{H}(x,p,A)=\lim\limits_{l\to\infty}\uparrow H^{M_0}(x,lp,lA),\ (x,p,A)\in \mathbb{R}^N\times\mathbb{R}^N\times\mathcal{S}^N.
\end{equation*}
Passing in (\ref{th3.4.5}) to the limit as $n\to \infty$ yields $\bar{w}(x)+H^{M_0}(x,l D\phi_\varepsilon(x),l D^2\phi_\varepsilon(x))\leq 0$, which in turn implies
$$\bar{w}(x)+\overline{H}(x,D\phi_\varepsilon(x),D^2\phi_\varepsilon(x))\leq 0,$$
by taking first the limit as $\varepsilon\to 0$ and after the supremum over $l>0$, i.e., we have shown (\ref{r19}).

Now, with the same $\phi$ and the same $x\in\theta$ as above, from (\ref{r19}) it follows that, for any $\lambda>0$,
$$\bar{w}(x)+H^{M_0}(x,\frac{1}{\lambda}D\phi(x),\frac{1}{\lambda}D^2\phi(x))\leq \bar{w}(x)+\overline{H}(x,D\phi(x),D^2D\phi(x))\leq 0,$$
i.e., $\bar{w}$ is a (continuous) viscosity subsolution of (\ref{rl7}) in $\theta$. Since $w_\lambda$ is a continuous constrained viscosity solution of (\ref{rl7}), from Theorem \ref{th1} (comparison principle) we have that
$$\bar{w}(x)\leq w_\lambda(x), \ \mbox{for\ all}\ x\in\overline{\theta}.$$
Taking the limit as $\lambda\rightarrow 0^+$ yields $w_0\geq \bar{w}$ and completes the proof.\\
\end{proof}
We now give an application of the above Theorem \ref{the:3.4}, which generalize the results in \cite{Marc 2015}.

For this recall that the second order superjet at $x\in\overline{\theta}$ for a function $u\in USC(\overline{\theta})$ is defined by
\begin{equation*}
  J^{2,+}u(x)=\{(D\varphi(x),D^2\varphi(x)): \varphi\in C^2(\mathbb{R}^N), u-\varphi\leq u(x)-\varphi(x)\ \mbox{on}\ \overline{\theta}\},
\end{equation*}
while, for $v\in LSC(\overline{\theta})$,
\begin{equation*}
  J^{2,-}v(x)=\{-(p,A)|(p,A)\in J^{2,+}(-v)(x)\}
\end{equation*}
defines the second order subjet.

\begin{corollary}\label{c1}
Under the same assumptions in Theorem \ref{the:3.4}, we have, for all $x\in\theta$,
\begin{equation*}
  \{(p,A)\in J^{2,+}w_0(x): \sup\limits_{l>0}H(x,lp,lA)=+\infty\}=\emptyset.
\end{equation*}
\end{corollary}
\begin{proof}
Assume that, for some $x\in\theta$, $(p,A)\in J^{2,+}w_0(x)$ and
\begin{equation}\label{b1}
 \sup\limits_{l>0}H(x,lp,lA)=+\infty.
\end{equation}
From (\ref{b1}), $\overline{H}(x,p,A)=M_0$. Furthermore, from (\ref{r18}), in viscosity sense,
\begin{equation*}
  w_0(x)+\overline{H}(x,Dw_0(x),D^2w_0(x))\leq0,\ x\in\theta.
\end{equation*}
Consequently, since $w_0\in Lip_{M_0}(\overline{\theta})$ we get $0\geq w_0(x)+\overline{H}(x,p,A)=w_0(x)+M_0$. Therefore, $w_0(x)\leq-M_0$. But, on the other hand $w_0\in Lip_{M_0}(\overline{\theta})$ implies $|w_0(x)|\leq M_0$. Hence, $w_0(x)=-M_0$. As for all $M\geq M_0$, $w_0\in\mathcal{S}_{H,M}$ and $w_0\in Lip_{M}(\overline{\theta})$, the same argument also gives $w_0(x)=-M$. This is a contradiction, and it follows that there cannot exist $(p,A)\in J^{2,+}w_0(x)$ with $\sup\limits_{l>0}H(x,lp,lA)=+\infty$.
\end{proof}
\begin{corollary}\label{c2}
Under the same assumptions in Theorem \ref{the:3.4} we suppose that, for all $x\in\theta, (p,A)\in(\mathbb{R}^N\backslash\{0\})\times\mathcal{S}^N$, $\sup\limits_{l>0}H(x,lp,lA)=+\infty$. Then, $w_0$ is a constant on $\overline{\theta}$.
\end{corollary}
\begin{proof}
Let $x_0\in \theta$ and $\varphi_\varepsilon\in C^2(\mathbb{R}^N)$, $\varepsilon>0$, be such that

i) $\varphi_\varepsilon(x)=\left\{
\begin{array}{lll}
\frac{\varepsilon}{2}|x-x_0|^2,\ x\in\theta\ \mbox{with\ dist}(x,\partial\theta)\geq\frac{1}{2} \mbox{dist}(x_0,\partial\theta)\ (>0),\\
\geq3M_0,\ x\in \theta^C,
\end{array}
\right.$

ii) $D{\varphi_\varepsilon}(x)\neq0,\ x\in\mathbb{R}^N\setminus\{x_0\},$

iii) $\varphi_\varepsilon(x)\to0\ (\varepsilon\downarrow0),$ for all $x\in\theta$.

\noindent For $\psi_\varepsilon(x)=w_0(x)-\varphi_\varepsilon(x)$, $x\in\overline{\theta}$, let $x_\varepsilon\in\overline{\theta}$ be the maximum point of $\psi_\varepsilon$. As for all $x'\in\partial\theta, \psi_\varepsilon(x')\leq w_0(x')-3M_0\leq-2M_0$ (Recall that $w_0\in \mbox{Lip}_{M_0}(\overline{\theta})$), and $\psi_\varepsilon(x_\varepsilon)\geq-M_0.$\ It follows that $x_\varepsilon\in\theta$. Since $(D\varphi_\varepsilon(x_\varepsilon),D^2\varphi_\varepsilon(x_\varepsilon))\in J^{2,+}w_0(x_\varepsilon)$, we get from Corollary \ref{c1} that $\sup\limits_{l>0}H(x_\varepsilon,lD\varphi_\varepsilon(x_\varepsilon),lD^2\varphi_\varepsilon(x_\varepsilon))<+\infty$.

Hence, from the assumptions of Corollary \ref{c2}, $D\varphi_\varepsilon(x_\varepsilon)=0$, i.e., $x_\varepsilon=x_0,\ \varepsilon>0$. Consequently,
$$w_0(x_0)=\psi_\varepsilon(x_0)\geq\psi_\varepsilon(x)=w_0(x)-\varphi_\varepsilon(x)\to w_0(x),\ \mbox{as}\ \varepsilon\to 0,\ \mbox{for\ all}\ x\in \theta.$$
Then, from the arbitrariness of $x_0\in\theta$ it follows that $w_0(x)=w_0(x_0)$, for all $x\in\theta$, and from the continuity of $w_0$ on $\overline{\theta}$ we, finally, have $w_0(x)=w_0(x_0), x\in\overline{\theta}$.

\end{proof}

\section{ {\protect \large Convergence problem for the optimal control}}
After a more general discussion in the previous section, in this part we consider the Hamiltonian $H$ of the form
\begin{equation}\label{r20}
  H(x,p,A):=\sup\limits_{u\in U}\{{\langle -p,b(x,u)\rangle}-\frac{1}{2}tr(\sigma\sigma^*(x,u)A)-\psi(x,p\sigma(x,u),u)\}, \ (x,p,A)\in\mathbb{R}^N\times\mathbb{R}^N\times\mathcal{S}^N.
\end{equation}
\begin{proposition}\label{th:3.3.1} Under the assumptions (\ref{r1}), (\ref{r48}) and (\ref{r49}) the value function $V_\lambda$ defined by (\ref{r94}) is a viscosity solution of the Hamilton-Jacobi-Bellman equation
\begin{equation}\label{r8}
\lambda V(x)+H(x,DV(x),D^2V(x))=0,\ x\in\overline{\theta},
\end{equation}
where $H(x,p,A)$ is defined by (\ref{r20}).
\end{proposition}
\begin{remark}
Proposition \ref{th:3.3.1} shows that $V_\lambda$ defined by (\ref{r94}) is in $Lip_{\frac{M_0}{\lambda}}(\overline{\theta})$ and it is a viscosity solution on $\overline{\theta}$ of (\ref{r8}), i.e., a super-but also a subsolution on $\overline{\theta}$. Thus, $V_\lambda$ is, in particular also a constrained viscosity solution of (\ref{r8}), but unlike a constrained viscosity solution, $V_\lambda$ also satisfies, in viscosity sense,
\begin{equation*}
  \lambda V_\lambda(x)+H(x,DV_\lambda(x),D^2V_\lambda(X))\leq0,\ x\in\partial\theta,
\end{equation*}
i.e., for all $x\in\partial\theta$, $\varphi\in C^2(\mathbb{R}^N)$ with $V-\varphi\leq V(x)-\varphi(x)$, it holds
\begin{equation*}
  \lambda V_\lambda(x)+H(x,D\varphi(x),D^2\varphi(x))\leq0.
\end{equation*}

\end{remark}

The proof of Proposition \ref{th:3.3.1} uses Peng's BSDE method developed in \cite{Peng 1997}. To prove the proposition we need the dynamic programming principle (DPP) and the following three lemmas based on the notion of stochastic backward semigroups introduced by Peng \cite{Peng 1997}.

Given the initial value $x$ at time $t=0$ of SDE (\ref{r2}), $u(\cdot)\in\mathcal{U}$ and $\eta\in L^2(\Omega,\mathcal{F}_t,\mathbb{P})$, we define a
stochastic backward semigroup: For given $\lambda>0,\ x\in\overline{\theta},\ u\in\mathcal{U},\ t\in\mathbb{R}_+$, we put
\begin{equation*}
  G_{s,t}^{\lambda,x,u}[\eta]:=Y_s^\eta,\ s\in[0,t],\ \eta\in L^2(\Omega,\mathcal{F}_t,\mathbb{P}),
\end{equation*}
where $(Y_s^\eta)_{s\in[0,t]}$ is the unique solution of the BSDE
 \begin{equation*}
   \left\{
\begin{array}{ll}
dY_s^\eta=-(\psi(X_s^{x,u},Z_s^\eta,u_s)-\lambda Y_s^\eta)ds+Z_s^\eta dW_s,\ s\in[0,t],\\
Y_t^\eta=\eta.
\end{array}
\right.
 \end{equation*}
\begin{proposition}\label{dpp}(DPP)
 Under the assumptions (\ref{r1}), (\ref{r48}) and (\ref{r49}), for all $\lambda>0,\ x\in\mathbb{R}^N$ and $t\geq0$, it holds
 \begin{equation*}
   V_\lambda(x)=\inf\limits_{u\in\mathcal{U}}G_{0,t}^{\lambda,x,u}[V_\lambda(X_t^{x,u})].
 \end{equation*}
\end{proposition}
The proof of the DPP uses arguments, which are for the case of control problems with finite time horizon rather standard by now (see, e.g., \cite{Buckdahn 2008}). But here the time horizon is infinite, and for convenience we give the proof in the Appendix. Let us give now three auxiliary lemmas.

For this, given a test function $\varphi\in C^3(\mathbb{R}^N)$, we define, for all $(x,y,z,u)\in\mathbb{R}^N\times\mathbb{R}\times\mathbb{R}^d\times U,$
\begin{equation*}
  \Phi(x,y,z,u):=\langle D\varphi(x),b(x,u)\rangle+\frac{1}{2}tr(\sigma\sigma^*(x,u)D^2\varphi(x))+\psi(x,z+D\varphi(x)\sigma(x,u),u)-\lambda\cdot(y+\varphi(x)).
\end{equation*}
Let $(Y_s^{1,u},Z_s^{1,u})_{s\in[0,t]}\in\mathcal{S}_{\mathbb{F}}^2([0,t];\mathbb{R})\times\mathcal{H}^2_{\mathbb{F}}([0,t];\mathbb{R}^d)\footnotemark[1]$ \footnotetext[1]{$\mathcal{S}_{\mathbb{F}}^2([0,t];\mathbb{R}):=\{\phi=(\phi_{s})_{s\in[0,t]}:(\phi_{s\wedge t})_{s\geq0}\in\mathcal{S}_{\mathbb{F}}^2(\mathbb{R})\}$, $\mathcal{H}^2_{\mathbb{F}}([0,t];\mathbb{R}^d):=\{\varphi=(\varphi_s)_{s\in[0,t]}:(\varphi_s1_{[0,t]}(s))_{s\geq0}\in\mathcal{H}^2_{\mathbb{F}}(\mathbb{R}^d)\}.$}
be the unique solution of the BSDE
\begin{equation}\label{r95}
\left\{
\begin{array}{ll}
dY_s^{1,u}=-\Phi(X_s^{x,u},Y_s^{1,u},Z_s^{1,u},u_s)ds+Z_s^{1,u}dW_s,\ s\in[0,t],\\
Y_t^{1,u}=0.
\end{array}
\right.
\end{equation}
As $\Phi(x,\cdot,\cdot,u)$ is Lipschitz, uniformly in $(x,u)$, and $\Phi(x,0,0,u)$ is bounded on $\overline{\theta}\times U$, the existence and the uniqueness are by now standard.
\begin{lemma}\label{lem1}
$Y_s^{1,u}=G_{s,t}^{\lambda,x,u}[\varphi(X_t^{x,u})]-\varphi(X_s^{x,u}),\ s\in[0,t].$
\end{lemma}
\begin{proof}
Notice that $G_{s,t}^{\lambda,x,u}[\varphi(X_t^{x,u})]$ is defined by the solution of the following BSDE
\begin{equation*}
  \left\{
  \begin{array}{ll}
dY_s^{\varphi}=-( \psi(X_s^{x,u},Z_s^{\varphi},u_s)-\lambda Y_s^\varphi)ds+Z_s^{\varphi}dW_s,\ s\in[0,t],\\
Y_t^{\varphi}=\varphi(X_t^{x,u}),
\end{array}
  \right.
\end{equation*}
that is
\begin{equation*}
  G_{s,t}^{\lambda,x,u}[\varphi(X_t^{x,u})]=Y_s^{\varphi},\ s\in[0,t].
\end{equation*}
We only need to prove that $Y_s^{\varphi}-\varphi(X_s^{x,u})= Y_s^{1,u}$, $s\in[0,t]$. Applying It\^{o}'s formula to $\varphi(X_s^{x,u})$, it is obvious that $d(Y_s^{\varphi}-\varphi(X_s^{x,u}))= dY_s^{1,u}$. As $Y_t^{\varphi}-\varphi(X_t^{x,u})=0=Y_t^{1,u}$, it follows that $Y_s^{\varphi}-\varphi(X_s^{x,u})= Y_s^{1,u}$, $s\in[0,t]$.
\end{proof}
Now we consider BSDE (\ref{r95}) in which $X_s^{x,u}$ is replaced by its initial condition $X_0^{x,u}=x$:
\begin{equation}\label{r96}
\left\{
\begin{array}{ll}
dY_s^{2,u}=-\Phi(x,Y_s^{2,u},Z_s^{2,u},u_s)ds+Z_s^{2,u}dW_s,\ s\in[0,t],\\
Y_t^{2,u}=0.
\end{array}
\right.
\end{equation}
It has a unique solution $(Y^{2,u},Z^{2,u})\in\mathcal{S}_{\mathbb{F}}^2([0,t];\mathbb{R})\times\mathcal{H}_{\mathbb{F}}^2([0,t];\mathbb{R}^d)$.
\begin{lemma}\label{lem2}
We have $|Y_0^{1,u}-Y_0^{2,u}|\leq ct^{\frac{3}{2}},\ \mbox{for all}\ t\in [0, T],\ u\in\mathcal{U}$, where $c\in\mathbb{R}_+$ is independent of $u\in\mathcal{U}$ and depends only on $T>0$.
\end{lemma}
\begin{proof}
As $\overline{\theta}\subset\mathbb{R}^N$ is compact, $\varphi, D\varphi$ and $D^2\varphi$ are bounded and Lipschitz on $\overline{\theta}$. Combined with the boundedness and the Lipschitz property of $b(\cdot,u), \sigma(\cdot,u)$ which is uniform with respect to $u\in U$, this has the consequence that $\overline{\theta}\ni x\rightarrow\Phi(x,y,z,u)$ is Lipschitz, uniformly in $(y,z,u)$. Then, using BSDE and SDE standard estimates, we get
\begin{equation*}
\begin{array}{lll}
&\displaystyle|Y_0^{1,u}-Y_0^{2,u}|^2\leq\mathbb{E}[|Y_0^{1,u}-Y_0^{2,u}|^2+\int_0^t|Z_s^{1,u}-Z_s^{2,u}|^2ds]\\
\leq&\displaystyle c\mathbb{E}[(\int_0^t|\Phi(X_s^{x,u},Y_s^{2,u},Z_s^{2,u},u_s)-\Phi(x,Y_s^{2,u},Z_s^{2,u},u_s)|ds)^2]\\
\leq&\displaystyle ct^2\mathbb{E}[\sup\limits_{s\in[0,t]}|X_s^{x,u}-x|^2]\leq ct^3.\\
\end{array}
\end{equation*}
\end{proof}
We now define $\overline{\Phi}(x,y,z):=\inf\limits_{u\in U}\Phi(x,y,z,u),\ (x,y,z)\in\overline{\theta}\times\mathbb{R}\times\mathbb{R}^d.$ Note that $\overline{\Phi}(x,y,z)=\overline{\Phi}(x,0,z)-\lambda y$ and that $(x,y,z)\rightarrow\overline{\Phi}(x,y,z)$ is Lipschitz. We consider the following ODE
\begin{equation}\label{r97}
\left\{
\begin{array}{ll}
  dY_s^0=-\overline{\Phi}(x,Y_s^0,0)ds,\ s\in[0,t],\\
  Y_t^0=0.
\end{array}
\right.
\end{equation}
\begin{remark}
As $\overline{\Phi}(x,y,z)=\overline{\Phi}(x,0,z)-\lambda y$, the unique solution of (\ref{r97}) is given by
\begin{equation*}
 Y_s^0=\int_s^te^{-\lambda(r-s)}dr\cdot\overline{\Phi}(x,0,0),\ s\in[0,t].
\end{equation*}
\end{remark}
\begin{lemma}\label{lem3}
 $Y_s^0=\essinf\limits_{u\in\mathcal{U}}Y_s^{2,u}, s\in[0,t]$, i.e., in particular, $Y_0^0=\inf\limits_{u\in\mathcal{U}}Y_0^{2,u}$.
\end{lemma}
\begin{proof}
From the comparison theorem for BSDEs we obtain easily that $Y_s^0\leq Y_s^{2,u},\ s\in[0,t]$, for all $u\in\mathcal{U}$. On the other hand, as $U$ is compact and $\Phi(x,0,0,\cdot)$ continuous on $U$, there is $u^*\in U$ such that $\overline{\Phi}(x,0,0)=\Phi(x,0,0,u^*)$. Then, for $u=(u_s)_{s\geq0}\in\mathcal{U}$ defined by $u_s=u^*,\ s\geq0$, $(Y^0,Z^0)=(Y^0,0)$ solves the BSDE
\begin{equation*}
  dY_s^{2,u}=-\Phi(x,Y_s^{2,u},Z_s^{2,u},u_s)ds+Z_s^{2,u}dW_s,\ s\in[0,t],\ Y_t^{2,u}=0,
\end{equation*}
and from the uniqueness of its solution we get $Y_s^{2,u}=Y_s^0,\ s\in[0,t]$. The proof is complete.
\end{proof}
Now we are able to give the proof of Proposition \ref{th:3.3.1}.
\begin{proof}\textbf{(of Proposition \ref{th:3.3.1}.)} From Lemma \ref{lem:2.6} we know that $V_\lambda\in C(\overline{\theta})$. Let $x\in {\theta}$ and $\varphi\in C^3(\mathbb{R}^N)$ be such that $0=(V_\lambda-\varphi)(x)\geq V_\lambda-\varphi$ on $\overline{\theta}$. Then, for all $u\in\mathcal{U}$ and $t>0,$ the DPP and the monotonicity of $G_{0,t}^{\lambda,x,u}[\cdot]$ yield:
\begin{equation*}
  \varphi(x)=V_\lambda(x)=\inf\limits_{u\in\mathcal{U}}G_{0,t}^{\lambda,x,u}[V_\lambda(X_t^{x,u})]\leq\inf\limits_{u\in\mathcal{U}}G_{0,t}^{\lambda,x,u}[\varphi(X_t^{x,u})].
\end{equation*}
Thus, due to Lemma \ref{lem1},
\begin{equation*}
 \inf\limits_{u\in\mathcal{U}}Y_0^{1,u}= \inf\limits_{u\in\mathcal{U}}(G_{0,t}^{\lambda,x,u}[\varphi(X_t^{x,u})]-\varphi(x))\geq0,
\end{equation*}
and Lemma \ref{lem2} implies together with Lemma \ref{lem3},
\begin{equation*}
  Y_0^0=\inf\limits_{u\in\mathcal{U}}Y_0^{2,u}\geq-ct^{\frac{3}{2}},\ \mbox{i.e.},\ \int_0^te^{-\lambda r}dr\cdot\overline{\Phi}(x,0,0)\geq-ct^{\frac{3}{2}},\ t>0.
\end{equation*}
Dividing the above relation by $t$ and taking after the limit as $t\downarrow0$, we get
\begin{equation*}
  0\leq\overline{\Phi}(x,0,0)=\inf\limits_{u\in U}\{\langle D\varphi(x),b(x,u)\rangle+\frac{1}{2}tr(\sigma\sigma^*(x,u)D^2\varphi(x))+\psi(x,D\varphi(x)\sigma(x,u),u)\}-\lambda\varphi(x),
\end{equation*}
i.e., $\lambda V_\lambda(x)+H(x,D\varphi(x),D^2\varphi(x))\leq0$.
This proves that $V_\lambda$ is a subsolution on $\overline{\theta}$; the proof that $V_\lambda$ is a viscosity supersolution on $\overline{\theta}$ is similar, and thus, omitted here.

As $V_\lambda\in Lip_{\frac{M_0}{\lambda}}(\overline{\theta})$ is a constrained viscosity solution of (\ref{r8}), we have from Theorem \ref{th1} (Comparison principle) its uniqueness in $C(\overline{\theta})$. However, for the convenience of the reader let us give the following comparison result and its proof for Hamiltonians of the form (\ref{r20}).
\end{proof}

\indent We have the uniqueness of the viscosity solution from the following theorem. For this we recall that $\overline{\theta}$  is a compact subset of $\mathbb{R}^N$ and invariant with respect to the control system (\ref{r2}).
\begin{proposition}\label{th:3.2} Assume (\ref{r1}) holds. Let $H_1,\ H_2:\mathbb{R}^N\times\mathbb{R}^N\times\mathcal{S}^N\rightarrow\mathbb{R}$ be two Hamiltonians of the form (\ref{r20}) with $\psi=\psi_1$ and $\psi=\psi_2$, respectively, where $\psi_1$ and $\psi_2$ are assumed to satisfy (\ref{r48}). We suppose that $u\in USC(\overline{\theta})$ is a subsolution of
\begin{equation*}
 \lambda V(x)+H_1(x,D\varphi(x),D^2\varphi(x))=0,\ x\in\overline{\theta},
 \end{equation*}
and $v\in LSC(\overline{\theta})$
is a supersolution of
\begin{equation*}
 \lambda V(x)+H_2(x,D\varphi(x),D^2\varphi(x))=0,\ x\in\overline{\theta}.
 \end{equation*}
Then it holds
\begin{equation*}
  \lambda (u(x)- v(x))\leq\sup\limits_{u\in U,  x\in\bar{\theta} \hfill \atop \scriptstyle z\in\mathbb{R}^d}\{|\psi_1(x,z,u)-\psi_2(x,z,u)|\},\ \mbox{for any}\ x\in\overline{\theta}.
\end{equation*}
\end{proposition}
\begin{proof}
Let $u\in USC(\overline{\theta})$ be a subsolution and $v\in LSC(\overline{\theta})$ a supersolution. For $\varepsilon>0$ arbitrarily chosen, we define $\Phi_\varepsilon(x,x'):=u(x)-v(x')-\frac{1}{2\varepsilon}|x-x'|^2$, $(x,x')\in\overline{\theta}\times\overline{\theta}$. Let $(x_\varepsilon,x'_\varepsilon)\in\overline{\theta}\times\overline{\theta}$ denote a maximum point of the USC-function $\Phi_\varepsilon $ on the compact set $\overline{\theta}\times\overline{\theta}$. We set $\varphi_\varepsilon(x,x')=\frac{1}{2\varepsilon}|x-x'|^2$. Then
$u(x)-\varphi_\varepsilon(x,x'_\varepsilon)$ attains a maximum at $x=x_\varepsilon$ and $v(x')+\varphi_\varepsilon(x_\varepsilon,x')$ attains a minimum at $x'=x'_\varepsilon$.

From Theorem 3.2 in \cite{ISHII 1992} we have the existence of two matrices $A, B\in\mathcal{S}^N$ with
\begin{equation*}
(\frac{x_\varepsilon-x'_\varepsilon}{\varepsilon},A)\in \overline{J}^{2,+}u(x_\varepsilon),\quad
(\frac{x_\varepsilon-x'_\varepsilon}{\varepsilon},B)\in \overline{J}^{2,-}v(x'_\varepsilon),
\end{equation*}
such that
\begin{gather}\label{r98}
\begin{pmatrix}
A&0\\
0&-B
\end{pmatrix}
\leq A_0+\varepsilon A_0^2,\
A_0=D^2\varphi_\varepsilon(x,x')=\frac{1}{\varepsilon}
\begin{pmatrix}
I&-I\\
-I&I
\end{pmatrix}.
\end{gather}
We notice that $A_0+\varepsilon A_0^2=\frac{3}{\varepsilon}
\begin{pmatrix}
I&-I\\
-I&I
\end{pmatrix}.$
Then, as $u\in USC(\overline{\theta})$ is a subsolution on $\overline{\theta}$ and $v\in LSC(\overline{\theta})$ a supersolution on $\overline{\theta}$,
\begin{equation}\label{t5}
  \lambda u(x_\varepsilon)+H_1(x_\varepsilon,\frac{x_\varepsilon-x'_\varepsilon}{\varepsilon},A)\leq0,\ \
  \lambda v(x'_\varepsilon)+H_2(x'_\varepsilon,\frac{x_\varepsilon-x'_\varepsilon}{\varepsilon},B)\geq0.
\end{equation}
We set $\beta:=\sup\limits_{x\in\bar{\theta}}(u(x)-v(x))$. As $\overline{\theta}$ is compact, and $u-v$ upper semicontinuous on $\overline{\theta}$, there exists $\overline{x}\in\overline{\theta}$\ such that $u(\overline{x})-v(\overline{x})=\beta$. Then $\Phi_\varepsilon(x_\varepsilon,x'_\varepsilon)\geq u(\overline{x})-v(\overline{x})=\beta.$\\
Obviously, since
\begin{equation*}
  \frac{|x_\varepsilon-x'_\varepsilon|^2}{2\varepsilon}=u(x_\varepsilon)-v(x'_\varepsilon)-\Phi_\varepsilon(x_\varepsilon,x'_\varepsilon)\leq u(x_\varepsilon)-v(x'_\varepsilon)-\beta\leq c,\ \varepsilon>0,
\end{equation*}
we have that $|x_\varepsilon-x'_\varepsilon|^2\leq c\varepsilon$, and by letting $\varepsilon\to 0$ we get $\lim\limits_{\varepsilon\downarrow0}|x_\varepsilon-x'_\varepsilon|=0.$

As $\overline{\theta}$ is compact, there exists a subsequence of $(x_\varepsilon,x'_\varepsilon)\in\overline{\theta}, \varepsilon>0$, again denoted by $(x_\varepsilon,x'_\varepsilon)$, and some $\widehat{x}\in\overline{\theta}$ such that
$x_\varepsilon\rightarrow\widehat{x},\ x'_\varepsilon\rightarrow\widehat{x}$ as $\varepsilon\downarrow0$.
Consequently, $$0\leq\varlimsup\limits_{\varepsilon\downarrow0}\frac{|x_\varepsilon-x'_\varepsilon|^2}{2\varepsilon}\leq u(\widehat{x})-v(\widehat{x})-\beta\leq0.$$
It follows that
\begin{equation*}
\left\{
\begin{array}{lll}
u(\widehat{x})-v(\widehat{x})=\beta=\max\limits_{x\in\overline{\theta}}(u(x)-v(x)),\\
\lim\limits_{\varepsilon\downarrow0}\frac{|x_\varepsilon-x'_\varepsilon|^2}{2\varepsilon}=0.
\end{array}
\right.
\end{equation*}
From (\ref{t5}) we have
\begin{equation*}
\begin{split}
  0\geq\lambda u(x_\varepsilon)+\sup\limits_{u\in U}\{-b(x_\varepsilon,u)\frac{x_\varepsilon-x'_\varepsilon}{\varepsilon}-\frac{1}{2}tr(\sigma\sigma^*(x_\varepsilon,u)A)-\psi_1(x_\varepsilon,\frac{x_\varepsilon-x'_\varepsilon}{\varepsilon}\sigma(x_\varepsilon,u),u)\},\\
  0\leq\lambda v(x'_\varepsilon)+\sup\limits_{u\in U}\{-b(x'_\varepsilon,u)\frac{x_\varepsilon-x'_\varepsilon}{\varepsilon}-\frac{1}{2}tr(\sigma\sigma^*(x'_\varepsilon,u)B)-\psi_2(x'_\varepsilon,\frac{x_\varepsilon-x'_\varepsilon}{\varepsilon}\sigma(x'_\varepsilon,u),u)\}.
  \end{split}
\end{equation*}
Then, combined with (\ref{r98}), we obtain by using the Lipschitz assumption on $b, \sigma, \psi_1$ and $\psi_2$,
\begin{equation*}
\begin{array}{lll}
  &\displaystyle\lambda (u(x_\varepsilon)-v(x'_\varepsilon))
  \leq\sup\limits_{u\in U}\{(b(x_\varepsilon,u)-b(x'_\varepsilon,u))\frac{x_\varepsilon-x'_\varepsilon}{\varepsilon}+\frac{1}{2}tr\big(\sigma\sigma^*(x_\varepsilon,u)A-\sigma\sigma^*(x'_\varepsilon,u)B\big)\\
  &\displaystyle\quad\quad\quad\quad\quad\quad\quad \quad\quad\quad\quad +(\psi_1(x_\varepsilon,\frac{x_\varepsilon-x'_\varepsilon}{\varepsilon}\sigma(x_\varepsilon,u),u)-\psi_2(x'_\varepsilon,\frac{x_\varepsilon-x'_\varepsilon}{\varepsilon}\sigma(x'_\varepsilon,u),u))\}\\
  &\displaystyle\leq c\big(\frac{|x_\varepsilon-x'_\varepsilon|^2}{\varepsilon}+\sup\limits_{u\in U}\{|\psi_1(x_\varepsilon,\frac{x_\varepsilon-x'_\varepsilon}{\varepsilon}\sigma(x_\varepsilon,u),u)-\psi_2(x'_\varepsilon,\frac{x_\varepsilon-x'_\varepsilon}{\varepsilon}\sigma(x'_\varepsilon,u),u)|\}\big)\\
  &\displaystyle\leq c(\frac{|x_\varepsilon-x'_\varepsilon|^2}{\varepsilon}+|x_\varepsilon-x'_\varepsilon|)+\sup\limits_{x\in\overline{\theta}, p\in\mathbb{R}^N, \hfill \atop \scriptstyle u\in U}\{|\psi_1(x,p\sigma(x,u),u)-\psi_2(x,p\sigma(x,u),u)|\}.
 \end{array}
\end{equation*}
Finally, letting $\varepsilon\downarrow0$, this yields
\begin{equation*}
  \lambda\max\limits_{x\in\overline{\theta}}(u(x)-v(x))=\lambda (u(\widehat{x})-v(\widehat{x}))\leq \sup\limits_{x\in\overline{\theta}, p\in\mathbb{R}^N, \hfill \atop \scriptstyle u\in U}\{|\psi_1(x,p\sigma(x,u),u)-\psi_2(x,p\sigma(x,u),u)|\}.
\end{equation*}
\end{proof}

\begin{theorem}\label{th:4.1}
We suppose that the assumptions (\ref{r1}), (\ref{r48}) and (\ref{r49}) hold. Moreover, we suppose:
\begin{equation*}\label{r100}
\begin{array}{lll}
\mbox{There\ is\ a\ concave\ increasing\ function}\ \rho:\mathbb{R}_+\to\mathbb{R}_+ \ \mbox{with}\ \rho(0+)=0\ \mbox{such\ that,\ for\ all}\ (x,z)\\ \tag{H6}
  \in\mathbb{R}^N\times\mathbb{R}^d,\ u,\ u'\in U,\  \mid \psi(x,z,u)-\psi(x,z,u')\mid\leq(1+|z|)\rho(d(u,u'))
\end{array}
\end{equation*}
(Recall that $d$\  is the metric we consider on the control state space $U$).
Then, along a suitable subsequence $0<\lambda_n\downarrow 0$, there exists the uniform limit $\widetilde{w}(x)=\lim\limits_{\lambda\to0^+}\lambda V_\lambda(x)$ (recall that $V_\lambda(x)$ is defined by (\ref{r94})) and it is a viscosity solution of the equation
$$h(x,D\widetilde{w}(x),D^2\widetilde{w}(x))=0,\ x\in \overline{\theta},$$
in the sense of Definition \ref{def3.1}, where $h(x,p,A)=\max\limits_{u\in U}\{\langle -p,b(x,u)\rangle-\frac{1}{2}tr(\sigma\sigma^*(x,u)A)-\widetilde{\psi}(p\sigma(x,u),u)\}$. The function $\widetilde{\psi}$ is described below in the proof.
\end{theorem}
\begin{proof}
Due to Proposition \ref{th:3.3.1} $V_\lambda(x)$ defined by (\ref{r94}) is a viscosity solution of
$$\lambda V(x)+H(x,DV(x),D^2V(x))=0\ \mbox{on}\ \overline{\theta}$$
(i.e., unlike a constrained viscosity solution $V_\lambda$ is a viscosity super-but also subsolution on $\overline{\theta}$),
and $\lambda V_\lambda\in\mbox{Lip}_{M_0}(\overline{\theta})$, for $M_0\geq \max\{\overline{c}_0,M\}$ (see Lemma \ref{lem:2.6}) and due to Proposition \ref{th:3.2} this viscosity solution is unique. We define $w_\lambda(x):=\lambda V_\lambda(x), x\in\overline{\theta}$. Then $w_\lambda$ is the unique viscosity solution of
\begin{equation}\label{226}\lambda w_\lambda(x)+H_\lambda(x,Dw_\lambda(x),D^2w_\lambda(x))=0\ \mbox{on}\  \overline{\theta},\end{equation} where
$$H_\lambda(x,p,A):=\lambda H(x,\frac{1}{\lambda} p,\frac{1}{\lambda} \lambda A)=\max\limits_{u\in U}\{\langle -p,b(x,u)\rangle-\frac{1}{2}tr(\sigma\sigma^*(x,u)A)-\lambda \psi(x,\frac{1}{\lambda}p\sigma(x,u),u)\}.$$
Due to (\ref{r48}) and (\ref{r100}) we have, for all $\lambda\in(0,1],\ (x,z,u),\ (x',z',u')\in\overline{\theta}\times\mathbb{R}^d\times U$,
\begin{equation*}
\begin{split}
&{\rm i)}\ | \lambda\psi(x,\frac{1}{\lambda}z,u)|\leq \lambda M+K_z|z|;\\
&{\rm ii)}\ | \lambda\psi(x,\frac{1}{\lambda}z,u)-\lambda\psi(x',\frac{1}{\lambda}z',u')|\leq \lambda K_x|x-x'|+K_z|z-z'|+(\lambda+|z|)\rho(d(u,u')),
\end{split}
\end{equation*}
i.e., combined with Lemma \ref{lem:2.6}, where we have shown that
\begin{equation*}
  |w_\lambda(x)|\leq M,\ x\in\overline{\theta};\ \ |w_\lambda(x)-w_\lambda(x')|\leq\overline{c}_0|x-x'|,\ x,\ x'\in\overline{\theta},\ \lambda>0,
\end{equation*}
we can apply the Arzel\'{a}-Ascoli Theorem to conclude that, for some sequence $\lambda_n\downarrow0$ (as $n\to\infty$), there are functions $\widetilde{w}:\overline{\theta}\to\mathbb{R},\ \  \widetilde{\psi}: \overline{\theta}\times\mathbb{R}^d\times U\to\mathbb{R}$ such that, for some $\widetilde{w}\in C(\overline{\theta})$, $w_{\lambda_n}\to\widetilde{w}$ ($n\to\infty$) uniformly on $\overline{\theta}$, and $\lambda_n\psi(x,\frac{1}{\lambda_n}z,u)\to\widetilde{\psi}(x,z,u)$ ($n\to\infty$), uniformly on compacts in $\overline{\theta}\times\mathbb{R}^d\times U$. Obviously,
\begin{equation*}
\begin{split}
&|\widetilde{w}(x)|\leq M,\ \ |\widetilde{w}(x)-\widetilde{w}(x')|\leq\overline{c}_0|x-x'|,\ x,\ x'\in\overline{\theta},\ \ \mbox{and}\\
&|\widetilde{\psi}(x,z,u)|\leq K_z|z|,\ \  |\widetilde{\psi}(x,z,u)-\widetilde{\psi}(x',z',u')|\leq K_z|z-z'|+|z|\rho(d(u,u')),
\end{split}
\end{equation*}
i.e., $\widetilde{\psi}(x,z,u)=\widetilde{\psi}(z,u),\ (x,z,u)\in\overline{\theta}\times\mathbb{R}^d\times U$, is independent of $x\in\overline{\theta}$.

Then, putting
\begin{equation*}
h(x,p,A):=\max\limits_{u\in U}\{\langle-p,b(x,u)\rangle-\frac{1}{2}tr(\sigma\sigma^*(x,u)A)-\widetilde{\psi}(p\sigma(x,u),u)\},
\end{equation*}
it follows that also $H_{\lambda_n}\to h\ (n\to\infty)$ uniformly on compacts. Finally, from (\ref{226}) and the stability result for viscosity solutions we see that $\widetilde{w}$ is a viscosity solution of the equation
\begin{equation*}
  h(x,D\widetilde{w}(x),D^2\widetilde{w}(x))=0,\ x\in\overline{\theta}.
\end{equation*}
\end{proof}
\begin{remark}
In Buckdahn, Li, Quincampoix \cite{Li} it is shown that the sequence $(w_\lambda)_{\lambda>0}$, as $\lambda\downarrow0$, can have at most only one accumulation point in the space $C(\overline{\theta})$ endowed with the supremum norm. As $\widetilde{w}$ is an accumulation point of $(w_\lambda)_{\lambda>0}$ and as due to Lemma \ref{lem:2.6} every subsequence of $w_\lambda$, $\lambda\downarrow0$, has a converging subsubsequence (Arzel\'{a}-Ascoli Theorem), it follows that $w_\lambda\to\widetilde{w} (\lambda\downarrow0)$, uniformly on $\overline{\theta}$. In particular, if we also suppose (\ref{r15}), we have $\widetilde{w}=w_0$.
\end{remark}
\begin{theorem}\label{th:4.2} We suppose that the assumptions (\ref{r1}), (\ref{r49}) and (\ref{r15}) hold true.
Now we consider the case: $\psi(x,z,u)=\psi_1(x,u)+g(z)$, where $\psi_1:\overline{\theta}\times U\rightarrow \mathbb{R}$ is bounded (by $M$), uniformly continuous and satisfies
\begin{equation*}
  |\psi_1(x,u)-\psi_1(x',u)|\leq K_x|x-x'|,\ \mbox{for any}\ x,\ x'\in\overline{\theta},\ u\in U,
\end{equation*}
while $g:\mathbb{R}^d\rightarrow\mathbb{R}$ is supposed to be Lipschitz (with Lipschitz constant $K_z$), positive homogeneous, concave and satisfies $g(0)=0$.
For $\eta\in L^2(\mathcal{F}_t)$, we consider the following BSDE
\begin{equation}\label{e1}
  Y_s^\eta=\eta+\int_s^tg(Z_r^\eta)dr-\int_s^tZ_r^\eta dW_r,\ s\in[0,t],
\end{equation}
and define the nonlinear expectation $\varepsilon^g[\eta]:=Y_0^\eta$. Then, there exists the uniform limit ${w}_0(x)=\lim\limits_{\lambda\to0^+}\lambda V_\lambda(x)$ (recall $V_\lambda(x)$ is defined by (\ref{r94})), and
\begin{equation*}
  w_0(x)=\inf\limits_{t\geq0, u\in\mathcal{U}}\varepsilon^g[\min\limits_{v\in U}\psi(X_t^{x,u},0,v)],\ \ \mbox{for any}\ x\in\overline{\theta}.
\end{equation*}
\end{theorem}
\begin{remark} {\rm (i)}\ $\varepsilon^g[.]$ is called $g$-expectation, it was first introduced by Peng, see, e.g., \cite{peng}. Its definition is independent of $t$. Indeed, if $\eta\in L^2(\mathcal{F}_s),\ s\leq t$, then, in (\ref{e1}), $Z_r^\eta=0,\ r\in[s,t]$.

{\rm (ii)} We recall the properties of $\varepsilon^g[\cdot]$, in particular, its concavity under the above assumptions on $g$: Let $\lambda_1, \lambda_2\in(0,1), \mbox{such\ that}\ \lambda_1+\lambda_2=1, \eta_1, \eta_2\in L^2(\mathcal{F}_t)$, $\overline{Y}_s:=(\lambda_1Y_s^{\eta_1}+\lambda_2Y_s^{\eta_2})-Y_s^{\lambda_1\eta_1+\lambda_2\eta_2}, \overline{Z}_s:=(\lambda_1Z_s^{\eta_1}+\lambda_2Z_s^{\eta_2})-Z_s^{\lambda_1\eta_1+\lambda_2\eta_2}$, $s\in[0,t]$. As the function $g$ is Lipschitz and concave, we get
\begin{equation*}
  \begin{split}
  &(\overline{Y}_s)^+(\lambda_1g(Z_s^{\eta_1})+\lambda_2g(Z_s^{\eta_2})-g(Z_s^{\lambda_1\eta_1+\lambda_2\eta_2}))\\
  \leq&(\overline{Y}_s)^+(g(\lambda_1Z_s^{\eta_1}+\lambda_2Z_s^{\eta_2})-g(Z_s^{\lambda_1\eta_1+\lambda_2\eta_2}))\\
  \leq&L(\overline{Y}_s)^+|\overline{Z}_s|, \ s\in[0,t].
  \end{split}
\end{equation*}
Hence, $\displaystyle\mathbb{E}[((\overline{Y}_s)^+)^2]+\mathbb{E}[\int_s^t|\overline{Z}_r|^21_{\{\overline{Y}_r>0\}}dr]\leq 2L\mathbb{E}[\int_s^t(\overline{Y}_s)^+|\overline{Z}_r|dr],\ s\in[0,t],$ and a standard estimate and Gronwall's inequality give $(\overline{Y}_s)^+=0$, i.e., $\lambda_1Y_s^{\eta_1}+\lambda_2Y_s^{\eta_2}\leq Y_s^{\lambda_1\eta_1+\lambda_2\eta_2},\ s\in[0,t],\ \mathbb{P}$-a.s. Thus, for $s=0$, $\varepsilon^g[\lambda_1\eta_1+\lambda_2\eta_2]\geq \lambda_1\varepsilon^g[\eta_1]+\lambda_2\varepsilon^g[\eta_2]$.\end{remark}
\begin{proof}\textbf{(of Theorem \ref{th:4.2}.)}\\
\textbf{Step 1}. From Proposition \ref{dpp} (DPP) we have $V_\lambda(x)=\inf\limits_{u\in\mathcal{U}}G_{0,t}^{\lambda,x,u}[V_\lambda(X_t^{x,u})]$,
where $G_{0,t}^{\lambda,x,u}[\eta]=\widetilde{Y}_{0,t}^{\lambda,x,u,\eta}$, for $\eta\in L^2(\mathcal{F}_t)$, defined by the BSDE
\begin{equation*}
  \left\{
\begin{array}{ll}
d\widetilde{Y}_{s,t}^{\lambda,x,u,\eta}=-(\psi(X_s^{x,u},\widetilde{Z}_{s,t}^{\lambda,x,u,\eta},u_s)-\lambda\widetilde{Y}_{s,t}^{\lambda,x,u,\eta})ds+\widetilde{Z}_{s,t}^{\lambda,x,u,\eta}dW_s,\\
\widetilde{Y}_{t,t}^{\lambda,x,u,\eta}=\eta,\ \eta:=V_\lambda(X_t^{x,u}).
\end{array}
\right.
\end{equation*}
Combined with the positive homogeneity of $g$, we obtain, for $s\in[0,t]$,
\begin{equation*}
  d(\lambda e^{-\lambda s}\widetilde{Y}_{s,t}^{\lambda,x,u,\eta}+\lambda\int_0^se^{-\lambda r}\psi_1(X_r^{x,u},u_r)dr)=-g(e^{-\lambda s}\lambda\widetilde{Z}_{s,t}^{\lambda,x,u,\eta})ds+e^{-\lambda s}\lambda\widetilde{Z}_{s,t}^{\lambda,x,u,\eta}dW_s.
\end{equation*}
On the other hand,
\begin{equation*}
  \lambda e^{-\lambda t}\widetilde{Y}_{t,t}^{\lambda,x,u,\eta}+\lambda\int_0^te^{-\lambda r}\psi_1(X_r^{x,u},u_r)dr=e^{-\lambda t}\lambda V_\lambda(X_t^{x,u})+\lambda\int_0^te^{-\lambda r}\psi_1(X_r^{x,u},u_r)dr.
\end{equation*}
Thus, for $\eta=V_\lambda(X_t^{x,u})$,
\begin{equation*}
  \varepsilon^g[e^{-\lambda t}\lambda V_\lambda(X_t^{x,u})+\lambda\int_0^te^{-\lambda r}\psi_1(X_r^{x,u},u_r)dr]=\lambda\widetilde{Y}_{0,t}^{\lambda,x,u,\eta}=\lambda G_{0,t}^{\lambda,x,u}[V_\lambda(X_t^{x,u})].
\end{equation*}
Hence,
\begin{equation}\label{f1}
 \lambda V_\lambda(x)=\inf\limits_{u\in\mathcal{U}}\lambda G_{0,t}^{\lambda,x,u}[V_\lambda(X_t^{x,u})]= \inf\limits_{u\in\mathcal{U}}\varepsilon^g[e^{-\lambda t}\lambda V_\lambda(X_t^{x,u})+\lambda\int_0^te^{-\lambda r}\psi_1(X_r^{x,u},u_r)dr].
\end{equation}
Notice that (see, e.g., \cite{peng}, or use just classical estimates for BSDE)
\begin{equation}\label{f2}
  \begin{split}
  &|\varepsilon^g[e^{-\lambda t}\lambda V_\lambda(X_t^{x,u})+\lambda\int_0^te^{-\lambda r}\psi_1(X_r^{x,u},u_r)dr]-\varepsilon^g[w_0(X_t^{x,u})]|\\
   \leq&c\parallel e^{-\lambda t}\lambda V_\lambda(X_t^{x,u})+\lambda\int_0^te^{-\lambda r}\psi_1(X_r^{x,u},u_r)dr-w_0(X_t^{x,u})\parallel_{L^2(\Omega)}\\
   \leq&c((1-e^{-\lambda t})+\parallel\lambda V_\lambda-w_0\parallel_\infty+\lambda tM),\ \mbox{for any}\ \lambda,\ t\geq0,\ u\in\mathcal{U}.
  \end{split}
\end{equation}
Thus, combining (\ref{f1}) and (\ref{f2}), we have
\begin{equation*}
 \lambda V_\lambda(x)= \inf\limits_{u\in\mathcal{U}}\varepsilon^g[w_0(X_t^{x,u})]+R_t^{\lambda,x},\ \text{with}\ |R_t^{\lambda,x}|\leq c((1-e^{-\lambda t})+\parallel \lambda V_\lambda-w_0\parallel_\infty+\lambda tM).
\end{equation*}
Then, letting $\lambda$ tend to 0 we get
\begin{equation}\label{d1}
  w_0(x)=\inf\limits_{u\in\mathcal{U}}\varepsilon^g[w_0(X_t^{x,u})],\ \mbox{for any}\ t\geq0,\ x\in\overline{\theta}.
\end{equation}
\textbf{Step 2.}\ From (\ref{f1}), using the monotonicity of $\varepsilon^g$ (resulting from the BSDE comparison theorem) and recalling that $|\lambda V_\lambda(x)|\leq M, \ \mbox{for\ all}\ x\in\overline{\theta},\ \lambda\geq0$, we obtain
\begin{equation*}
 \lambda V_\lambda(x)\geq \inf\limits_{u\in\mathcal{U}}\varepsilon^g[-Me^{-\lambda t}+\lambda\int_0^te^{-\lambda r}\min\limits_{v\in U}\psi_1(X_r^{x,u},v)dr].
\end{equation*}
Similar to (\ref{f2}) we get
\begin{equation*}
  \begin{split}
  &\sup\limits_{u\in\mathcal{U}}|\varepsilon^g[-Me^{-\lambda t}+\lambda\int_0^te^{-\lambda r}\min\limits_{v\in U}\psi_1(X_r^{x,u},v)dr]-\varepsilon^g[\lambda\int_0^\infty e^{-\lambda r}\min\limits_{v\in U}\psi_1(X_r^{x,u},v)dr]|\\
  \leq& c(Me^{-\lambda t}+\lambda\int_t^\infty e^{-\lambda r}dr\cdot M)=2cMe^{-\lambda t}\xrightarrow[t\uparrow+\infty]{}0.
  \end{split}
\end{equation*}
Consequently, using the concavity of $\varepsilon^g[\cdot]$ (see Remark 4.3.) this yields
\begin{equation*}
  \begin{split}
  &\lambda V_\lambda(x)\geq \inf\limits_{u\in\mathcal{U}}\varepsilon^g[\lambda\int_0^\infty e^{-\lambda r}\cdot\min\limits_{v\in U}\psi_1(X_r^{x,u},v)dr]-2cMe^{-\lambda t}\\
  &\geq \inf\limits_{u\in\mathcal{U}}\lambda\int_0^\infty e^{-\lambda r}\varepsilon^g[\min\limits_{v\in U}\psi_1(X_r^{x,u},v)]dr-2cMe^{-\lambda t}\\
  &\geq \inf\limits_{t\geq 0, u\in\mathcal{U}}\varepsilon^g[\min\limits_{v\in U}\psi_1(X_t^{x,u},v)]-2cMe^{-\lambda t},\ t\geq 0,\ x\in\overline{\theta}.
  \end{split}
\end{equation*}
Taking the limit as $t\rightarrow +\infty$ we get immediately
\begin{equation}\label{1000}\lambda V_\lambda(x)\geq \inf\limits_{t\geq0, u\in\mathcal{U}}\varepsilon^g[\min\limits_{v\in U}\psi_1(X_t^{x,u},v)],\ \mbox{for any}\ x\in\overline{\theta}.\end{equation}
 Combining Propositions \ref{th:3.3.1} and \ref{th:3.2} (comparison result), (\ref{r15}) allows to use the proof of Theorem \ref{the:3.4} without using ($A_\theta)$ (($A_H$) and (H) are satisfied since Hamiltonian $H$ is of the form (\ref{r20})) that $\lambda V_\lambda\to w_0$ uniformly on $\overline{\theta}$ as $\lambda\downarrow0$, and $w_0$ is the maximal viscosity subsolution on $\overline{\theta}$ of
\begin{equation}\label{r4.13}
  w_0(x)+\overline{H}(x,Dw_0(x),D^2w_0(x))\leq0,\ x\in\theta.
 \end{equation}Hence, letting $\lambda\downarrow0$ in above inequality (\ref{1000}) yields
 $$w_0(x)\geq \inf\limits_{t\geq0, u\in\mathcal{U}}\varepsilon^g[\min\limits_{v\in U}\psi_1(X_t^{x,u},v)],\ \mbox{for any}\ x\in\overline{\theta}.$$
\textbf{Step 3.}\ Recall (\ref{r4.13}). Then, for all $x\in\theta,\ (p,A)\in J^{2,+}w_0(x)$ thanks to (\ref{r15}) (see also Lemma \ref{l:3.4}) we have
\begin{equation*}
  0\geq w_0(x)+\overline{H}(x,p,A)\geq w_0(x)+\overline{H}(x,0,0)=w_0(x)+\sup\limits_{u\in U}(-\psi(x,0,u)).
\end{equation*}
This shows that, if $J^{2,+}w_0(x)\neq\emptyset$, then
\begin{equation}\label{d2}
  w_0(x)\leq\min\limits_{v\in U}\psi(x,0,v).
\end{equation}
Let $x\in\theta$\ and $\varepsilon>0$, and define
\begin{equation*}
  \psi_\varepsilon(y):=w_0(y)-\frac{1}{2\varepsilon}|y-x|^2,\ y\in\overline{\theta}.
\end{equation*}
Let $y_\varepsilon\in\overline{\theta}$ be a maximum point of $\psi_\varepsilon$. As $\psi_\varepsilon(y_\varepsilon)\geq\psi_\varepsilon(x)=w_0(x)$,\  $\frac{1}{2\varepsilon}|y_\varepsilon-x|^2\leq w_0(y_\varepsilon)-w_0(x)\leq2M$, we get $y_\varepsilon\xrightarrow[\varepsilon\downarrow0]{}x$, i.e., for $\varepsilon>0$ small enough, $y_\varepsilon\in\theta$. On the other hand, we have $(p,A):=(\frac{y_\varepsilon-x}{\varepsilon}, \frac{1}{\varepsilon}I_{\mathbb{R}^N})\in J^{2,+}w_0(y_\varepsilon)$.

From (\ref{d2}) we have $w_0(y_\varepsilon)\leq\min\limits_{v\in U}\psi(y_\varepsilon,0,v)$, and taking $\varepsilon\downarrow0$ yields $w_0(x)\leq \min\limits_{v\in U}\psi(x,0,v),$ $ \mbox{for any}\ x\in\theta$, and by the continuity of both sides of the inequality in $x\in\overline{\theta}$ we have
\begin{equation*}
  w_0(x)\leq \min\limits_{v\in U}\psi(x,0,v),\ \mbox{for all}\ x\in\overline{\theta}.
\end{equation*}
Finally, it follows from (\ref{d1}) and the monotonicity of $\varepsilon^g[\cdot]$ that $$
  w_0(x)\leq \inf\limits_{u\in\mathcal{U}}\varepsilon^g[\min\limits_{v\in U}\psi(X_t^{x,u},0,v)],\ t\geq0,\ x\in\overline{\theta},$$
which means $w_0(x)\leq \inf\limits_{t\geq0, u\in\mathcal{U}}\varepsilon^g[\min\limits_{v\in U}\psi(X_t^{x,u},0,v)],\ x\in\overline{\theta}.$ Combined with Step 2 we get
\begin{equation*}
  w_0(x)=\inf\limits_{t\geq0, u\in\mathcal{U}}\varepsilon^g[\min\limits_{v\in U}\psi(X_t^{x,u},0,v)],\ \mbox{for any}\ x\in\overline{\theta}.
\end{equation*}\end{proof}
\begin{remark}
Let us consider the special case where $\psi$ is independent of $z$, i.e., $g(z)=0$. Then we get $w_0(x)=\inf\limits_{t\geq0, u\in\mathcal{U}}\mathbb{E}[\min\limits_{v\in U}\psi(X_t^{x,u},0,v)]$, for any $x\in\overline{\theta}$.
\end{remark}
Let us come back now to a general case of $\psi$.
\begin{theorem}\label{th:4.3} We suppose that the assumptions (\ref{r1}), (\ref{r48}), (\ref{r49}), (\ref{r15}), (\ref{r100}) and\textbf{ ($A_\theta$)} hold true. Moreover, let $H(x,p,A)$ be convex in $(p,A)\in\mathbb{R}^N\times\mathcal{S}^N$, for all $x\in\overline{\theta}$. Then, we have
\begin{equation*}
  \begin{aligned}
w_0(x)\leq &\inf\{G_{0,t}^{\widetilde{\psi},x,u}[\min\limits_{v\in U}\psi(X_t^{x,u},0,v)]\mid u\in\mathcal{U},\ t\geq0,\ \widetilde{\psi}\ \mbox{such that there exists}\ \lambda_n\downarrow0\ \mbox{with}\\
&\ \  \ \ \lambda_n\psi(x,\frac{1}{\lambda_n}z,u)\rightarrow\widetilde{\psi}(z,u)\},\ x\in\overline{\theta}.
\end{aligned}
\end{equation*}
\end{theorem}
\begin{proof}
From Propositions \ref{th:3.3.1} and \ref{th:3.2} we get\\
1) The limit $w_0(x)=\lim_{\lambda\rightarrow0^+}\lambda V_\lambda(x)$, for every $x\in\overline{\theta}$; and the convergence is uniform on $\overline{\theta}$.\\
2) There exists $\widetilde{\psi}$\ such that $\lambda_n\psi(x,\frac{1}{\lambda_n}z,u)\to\widetilde{\psi}(z,u)$\ as $\lambda_n\downarrow0$, uniformly on compacts. Moreover,
\begin{equation*}
|\lambda_n\psi(x,\frac{1}{\lambda_n}z,u)| \leq \lambda_nM+K_z|z|,\ n\geq1,\quad |\widetilde{\psi}(z,u)|\leq K_z|z|,\ z\in\mathbb{R}^d.
\end{equation*}
From Proposition \ref{dpp} (DPP) we have
\begin{equation*}
  V_\lambda(x)=\inf_{u\in\mathcal{U}}G_{0,t}^{\lambda,x,u}[V_\lambda(X_t^{x,u})],\ t>0,\ x\in \overline{\theta}.
\end{equation*}
We put $\widetilde{Y}_s^{\lambda,x,u}:=G_{s,t}^{\lambda,x,u}[V_\lambda(X_t^{x,u})],\ s\in[0,t].$\ Then $
  V_\lambda(x)=\inf\limits_{u\in\mathcal{U}}\widetilde{Y}_0^{\lambda,x,u},$
where
\begin{equation*}
  \widetilde{Y}_s^{\lambda,x,u}=V_\lambda(X_t^{x,u})+\int_s^t(\psi(X_r^{x,u},\widetilde{Z}_r^{\lambda,x,u},u_r)
  -\lambda\widetilde{Y}_r^{\lambda,x,u})dr-\int_s^t\widetilde{Z}_r^{\lambda,x,u}dW_r,\ s\in[0,t].
\end{equation*}
By applying It\^{o}'s formula to $e^{-\lambda s}\widetilde{Y}_s^{\lambda,x,u}$ we have
\begin{equation*}
e^{-\lambda s}\lambda\widetilde{Y}_s^{\lambda,x,u}=e^{-\lambda t}\lambda\widetilde{Y}_t^{\lambda,x,u}+\int_s^t\lambda e^{-\lambda r}\psi(X_r^{x,u},\widetilde{Z}_r^{\lambda,x,u},u_r)dr-\int_s^t\lambda e^{-\lambda r}\widetilde{Z}_r^{\lambda,x,u}dW_r.
\end{equation*}
As $e^{-\lambda t}\lambda V_\lambda(X_t^{x,u})\xrightarrow[L^\infty]{} w_0(X_t^{x,u})$, as $\lambda\rightarrow 0$, uniformly with respect to $(t,x),\  u\in\mathcal{U}$, we consider the following BSDE:
\begin{equation}\label{326}
  Y_s^{x,u}=w_0(X_t^{x,u})+\int_s^t\widetilde{\psi}(Z_r^{x,u},u_r)dr-\int_s^tZ_r^{x,u}dW_r,\ s\in[t,T].
\end{equation}
From a standard estimate for BSDEs it follows that, for all $p\in(1,2)$,
\begin{equation*}
\begin{split}
  &\mathbb{E}[\sup\limits_{s\in[0,t]}|e^{-\lambda s}(\lambda\widetilde{Y}_s^{\lambda,x,u})-Y_s^{x,u}|^p+(\int_0^t|e^{-\lambda s}(\lambda\widetilde{Z}_s^{\lambda,x,u})-Z_s^{x,u}|^2ds)^{\frac{p}{2}}]\\
  \leq& C_p\mathbb{E}[|e^{-\lambda t}(\lambda\widetilde{Y}_t^{\lambda,x,u})-Y_t^{x,u}|^p]+C_p\mathbb{E}[(\int_0^t|e^{-\lambda r}\lambda\psi(X_r^{x,u},\widetilde{Z}_r^{\lambda,x,u},u_r)-\widetilde{\psi}(Z_r^{x,u},u_r)|dr)^p]\\
  \leq& C_p\mathbb{E}[|e^{-\lambda t}(\lambda\widetilde{Y}_t^{\lambda,x,u})-Y_t^{x,u}|^p]\\
  &+C_p\mathbb{E}[(\int_0^t\lambda|\psi(X_r^{x,u},\widetilde{Z}_r^{\lambda,x,u},u_r)-\psi(X_r^{x,u},\frac{1}{\lambda}{Z}_r^{x,u},u_r)|dr)^p] (=:I_1(\lambda))\\
  &+C_p\mathbb{E}[(\int_0^t|e^{-\lambda r}(\lambda\psi(X_r^{x,u},\frac{1}{\lambda}{Z}_r^{x,u},u_r))-\widetilde{\psi}(Z_r^{x,u},u_r)|1_{\{|Z_r^{x,u}|\leq\alpha\}}dr)^p] (=:\rho_\alpha(\lambda))\\
  &+C_p\mathbb{E}[(\int_0^t(\lambda M+2K_z|Z_r^{x,u}|)1_{\{|Z_r^{x,u}|>\alpha\}}dr)^p] (=:I_2(\lambda,\alpha)).
  \end{split}
\end{equation*}
Notice that $\displaystyle I_2(\lambda,\alpha)\leq C_{p,M}\lambda^p+C_{p,K_z}\mathbb{E}[\int_0^t\frac{|Z_r^{x,u}|^2}{\alpha^{2-p}}dr]$.
As $w_0\in C_b(\overline{\theta})$ and $|\widetilde{\psi}(z,u)|\leq K_z|z|,\ (z,u)\in\mathbb{R}^d\times U$, it follows from the BSDE (\ref{326}) for $(Y^{x,u},Z^{x,u})$ that
\begin{equation*}
  \sup\limits_{(x,u)\in\overline{\theta}\times\mathcal{U}}\mathbb{E}[\int_0^t|Z_r^{x,u}|^2dr]<\infty.
\end{equation*}
We also remark that, again by a BSDE standard estimates, there is some $K\in\mathbb{R}_+$ such that
\begin{equation*}
  \lambda^2\mathbb{E}[\int_0^t|\widetilde{Z}_r^{\lambda,x,u}|^2dr]\leq K,\ \mbox{for\ all}\ \lambda>0.
\end{equation*}
Then,
\begin{equation*}
  I_2(\lambda,\alpha)\leq C_{p,M}\lambda^p+C'_{p,K_z}\frac{1}{\alpha^{2-p}}.
\end{equation*}
For $I_1(\lambda)$ we have
\begin{equation*}
\begin{split}
  I_1(\lambda)\leq& C'_p\mathbb{E}[(\int_0^t|\lambda\widetilde{Z}_r^{\lambda,x,u}-Z_r^{x,u}|dr)^p]\\
  \leq&C''_pt^{\frac{p}{2}}\mathbb{E}[(\int_0^t|e^{-\lambda r}(\lambda\widetilde{Z}_r^{\lambda,x,u})-Z_r^{x,u}|^2dr)^{\frac{p}{2}}]+C'''_p\mathbb{E}[(\int_0^t(1-e^{-\lambda r})^2|\lambda\widetilde{Z}_r^{\lambda,x,u}|^2dr)^{\frac{p}{2}}]\\
  \leq& C''_pt^{\frac{p}{2}}\mathbb{E}[(\int_0^t|e^{-\lambda r}(\lambda\widetilde{Z}_r^{\lambda,x,u})-Z_r^{x,u}|^2dr)^{\frac{p}{2}}]+C'''_p(1-e^{-\lambda t})^pK.
  \end{split}
\end{equation*}
Hence, for $t>0$ small enough such that $C''_pt^{\frac{p}{2}}\leq\frac{1}{2}$, we get
\begin{equation*}
\begin{split}
  &\mathbb{E}[\sup\limits_{s\in[0,t]}|e^{-\lambda s}(\lambda\widetilde{Y}_s^{\lambda,x,u})-Y_s^{x,u}|^p+\frac{1}{2}(\int_0^t|e^{-\lambda s}(\lambda\widetilde{Z}_s^{\lambda,x,u})-Z_s^{x,u}|^2ds)^{\frac{p}{2}}]\\
  \leq& C\rho_\alpha(\lambda)+\frac{C}{\alpha^{2-p}},\  \mbox{for any}\ (x,u)\in\overline{\theta}\times\mathcal{U},\
   \mbox{and any}\ \alpha>0.
\end{split}
\end{equation*}
Observe that $\rho_\alpha(\lambda)\xrightarrow[\lambda=\lambda_n\downarrow0]{}0$, $\frac{C}{\alpha^{2-p}}\xrightarrow[\alpha\uparrow\infty]{}0$.
Hence,
\begin{equation*}
  \mathbb{E}[\sup\limits_{s\in[0,t]}|e^{-\lambda s}(\lambda\widetilde{Y}_s^{\lambda,x,u})-Y_s^{x,u}|^p+\frac{1}{2}(\int_0^t|e^{-\lambda s}(\lambda\widetilde{Z}_s^{\lambda,x,u})-Z_s^{x,u}|^2ds)^{\frac{p}{2}}]\xrightarrow[\lambda=\lambda_n\downarrow0]{}0,
\end{equation*}
 uniformly in $(x,u)\in\overline{\theta}\times\mathcal{U}$, for $t>0$ small enough; otherwise, for $\delta>0$ small enough, by making above discussion first on $[t-\delta,t]$, after on $[t-2\delta,t-\delta]$, etc., we get by iteration
\begin{equation*}
  \sup\limits_{u\in\mathcal{U}}|\lambda\widetilde{Y}_0^{\lambda,x,u}-Y_0^{x,u}|\xrightarrow[\lambda=\lambda_n\downarrow0]{}0,
\end{equation*}
and, consequently,
\begin{equation*}
  |\inf\limits_{u\in\mathcal{U}}(\lambda\widetilde{Y}_0^{\lambda,x,u})-\inf\limits_{u\in\mathcal{U}}Y_0^{x,u}|\xrightarrow[\lambda=\lambda_n\downarrow0]{}0.
\end{equation*}
But this means that
\begin{equation*}
  \inf\limits_{u\in\mathcal{U}}Y_0^{x,u}=w_0(x).
\end{equation*}
Notice that from BSDE (\ref{326}) we have
\begin{equation*}
  Y_0^{x,u}=w_0(X_t^{x,u})+\int_0^t\widetilde{\psi}(Z_s^{x,u},u_s)ds-\int_0^tZ_s^{x,u}dW_s.
\end{equation*}
On the other hand, defining the backward stochastic semigroup
\begin{equation*}
  G_{s,t}^{\widetilde{\psi},x,u}(\eta):=Y_s^{x,u,\eta}
\end{equation*}
through the associated BSDE
\begin{equation*}
  Y_s^{x,u,\eta}=\eta+\int_s^t\widetilde{\psi}(Z_r^{x,u,\eta},u_r)dr-\int_s^tZ_r^{x,u,\eta}dW_r,\ \eta\in L^2(\Omega,\mathcal{F}_t,\mathbb{P}),
\end{equation*}
we get
\begin{equation*}
  w_0(x)=\inf\limits_{u\in\mathcal{U}}G_{0,t}^{\widetilde{\psi},x,u}[w_0(X_t^{x,u})].
\end{equation*}
Consequently,
\begin{equation}\label{426}
  \begin{aligned}
w_0(x)= &\inf\{G_{0,t}^{\widetilde{\psi},x,u}[w_0(X_t^{x,u})]\mid u\in\mathcal{U},\ t\geq0,\ \widetilde{\psi}\ \mbox{such that there exists}\ \lambda_n\downarrow0\ \mbox{with}\\
&\ \  \ \ \lambda_n\psi(x,\frac{1}{\lambda_n}z,u)\rightarrow\widetilde{\psi}(z,u)\},\ x\in\overline{\theta}.
\end{aligned}
\end{equation}
From Lemma \ref{l:3.4} we have $H(x,p,A)\geq H(x,0,0),\ \mbox{for any}\ (p,A)\in \mathbb{R}^N\times\mathcal{S}^N, x\in\overline{\theta}.$

Therefore, from Proposition \ref{th:3.3.1}, in viscosity sense
\begin{equation*}
\begin{split}
  0\geq&\lambda V_\lambda(x)+H(x,DV_\lambda(x),D^2V_\lambda(x))\\
  \geq& \lambda V_\lambda(x)+H(x,0,0)=\lambda V_\lambda(x)+\max\limits_{u\in {U}}\{-\psi(x,0,u)\},
\end{split}
\end{equation*}
for all $x\in\theta$ with $J^{2,+}V_\lambda(x)\neq0$. Using the same argument as in the proof of Step 2 of Theorem \ref{th:4.1}, we see that this implies that $0\geq\lambda V_\lambda(x)+\max\limits_{u\in U}\{-\psi(x,0,u)\}$, for all $x\in\overline{\theta}$. By taking the limit, as $\lambda\rightarrow 0$, it follows that
\begin{equation*}
  w_0(x)\leq \min\limits_{u\in {U}}\psi(x,0,u).
\end{equation*}
Therefore, from (\ref{426}) and the comparison theorem for BSDEs we get directly
\begin{equation*}
  \begin{aligned}
w_0(x)\leq &\inf\{G_{0,t}^{\widetilde{\psi},x,u}[\min\limits_{v\in U}\psi(X_t^{x,u},0,v)]\mid u\in\mathcal{U},\ t\geq0,\ \widetilde{\psi}\ \mbox{such that there exists}\ \lambda_n\downarrow0\ \mbox{with}\\
&\ \  \ \ \lambda_n\psi(x,\frac{1}{\lambda_n}z,u)\rightarrow\widetilde{\psi}(z,u)\},\ x\in\overline{\theta}.
\end{aligned}
\end{equation*}
\end{proof}
\section{ {\protect \large Appendix: Proof of Proposition \ref{dpp} (DPP)}}
This appendix is devoted to the proof of the DPP (Proposition \ref{dpp}). For the proof we need an auxiliary result. For this we note that, as the filtration used in Section 4 is the Brownian one, we can suppose without loss of generality that $(\Omega,\mathcal{F},\mathbb{P})$ is the standard Wiener space, $\Omega=C_0(\mathbb{R}_+;\mathbb{R}^d)=\{\omega\in C(\mathbb{R}_+;\mathbb{R}^d): \omega(0)=0\}$ , endowed with Borel $\sigma$-algebra over $C_0(\mathbb{R}_+;\mathbb{R}^d)$ and the Wiener measure, with respect to which $\mathbb{F}$ is completed. The coordinate process $W_t(\omega)=\omega_t, t\geq0, \omega\in\Omega$, is a d-dimensional Brownian motion, and the filtration $\mathbb{F}$ is generated by $W$.
\begin{lemma} We assume that (H1) and (H2) hold. Let $t\geq0,\ u\in\mathcal{U}_t=L_{\mathbb{F}}^\infty(t,\infty;U)$. Let $X^{t,x,u}$ be the unique continuous and  $\mathbb{F}$-adapted solution of the following SDE:
\begin{equation*}\label{A1}
 X_s^{t,x,u}=x+\int_t^sb(X_r^{t,x,u},u_r)dr +\int_t^s\sigma(X_r^{t,x,u},u_r)dW_r,\ s\geq t,\ x\in\mathbb{R}^N,  \tag{A1}
\end{equation*}
and let $(Y^{\lambda,t,x,u}, Z^{\lambda,t,x,u})$ be the unique solution of the following BSDE on the infinite time inteval:
\begin{equation*}\label{A2}
  Y_s^{\lambda,t,x,u}=Y_T^{\lambda,t,x,u}+\int_s^T(\psi(X_r^{t,x,u},Z_r^{\lambda,t,x,u},u_r)-\lambda Y_r^{\lambda,t,x,u})dr-\int_s^TZ_r^{\lambda,t,x,u}dW_r,\ t\leq s\leq T<+\infty,  \tag{A2}
\end{equation*}
where $Y^{\lambda,t,x,u}=(Y_s^{\lambda,t,x,u})_{s\geq t}$ is a bounded continuous $\mathbb{F}$-adapted process and $Z^{\lambda,t,x,u}=(Z_s^{\lambda,t,x,u})_{s\geq t}\in\mathcal{H}^2_{loc}(t,\infty;\mathbb{R}^d)$.

Let $\theta_t=\theta_t(\omega)$ be the translation operator on $\Omega$, $\theta_t(\omega)_s=\omega(s+t)-\omega(t),\ \omega\in\Omega, s\geq t$. Given $u\in\mathcal{U}$ we can identify $u$ with a measurable functional applying to $W$. Thus, given an arbitrary element $u_0$ of $\mathcal{U}$, we can define
\begin{equation*}
\overline{u}_s:=
\left\{
\begin{array}{lll}
u_0,\ s\in[0,t),\\
u_{s-t}(\theta_t),\ s\geq t.
\end{array}
\right.
\end{equation*}
Then, $\overline{u}\in\mathcal{U}$ and
\begin{equation*}\label{A3}
  X_s^{x,u}(\theta_t)=X_{s+t}^{t,x,\overline{u}},\ Y_s^{\lambda,x,u}(\theta_t)=Y_{s+t}^{\lambda,t,x,\overline{u}},\ s\geq0,\ \mathbb{P}\text{-a.s.}, \tag{A3}
\end{equation*}
and
\begin{equation*}\label{A4}
  Z_s^{\lambda,x,u}(\theta_t)=Z_{s+t}^{\lambda,t,x,\overline{u}},\ dsd\mathbb{P}\text{-a.e.}, \ s\geq 0. \tag{A4}
\end{equation*}
\end{lemma}
\begin{proof}
While the existence and the uniqueness of the solution for (\ref{A1}) is standard, that of (\ref{A2}) is shown in analogy to Proposition \ref{th:2.4}.

Given $u\in\mathcal{U}$, it is obvious that also $\overline{u}\in\mathcal{U}$, and applying the transformation $\theta_t$ to (\ref{0}) and (\ref{r40}) we see that $(X_{s-t}^{x,u}(\theta_t))_{s\geq t}$, $(Y_{s-t}^{\lambda,x,u}(\theta_t),Z_{s-t}^{\lambda,x,u}(\theta_t))_{s\geq t}$ are solution of (\ref{A1}) and (\ref{A2}) respectively, with control process $\overline{u}$ instead of $u$. From the uniqueness of the solutions of (\ref{A1}) and (\ref{A2}) we obtain (\ref{A3}) and (\ref{A4}).
\end{proof}
Now we can prove Proposition \ref{dpp} (DPP).
\begin{proof}\textbf{(of Proposition \ref{dpp}.)} Let us put $\overline{V}_\lambda(x):=\inf\limits_{u\in\mathcal{U}}G_{0,t}^{\lambda,x,u}[V_\lambda(X_t^{x,u})]$.
We have to show that $\overline{V}_\lambda(x)=V_\lambda(x)$.\\
1) As, for all $y\in\mathbb{R}^d$, ${V}_\lambda(y)$ is deterministic, we obtain from the preceding Lemma 5.1
\begin{equation*}
  V_\lambda(y)=\essinf\limits_{v\in\mathcal{U}}Y_0^{\lambda,y,v}(\theta_t)=\essinf\limits_{v\in\mathcal{U}}Y_t^{\lambda,t,y,\overline{v}},\  \mathbb{P}\text{-a.s.},
\end{equation*}
with
\begin{equation*}
\overline{v}_s:=
\left\{
\begin{array}{lll}
u_0,\ s\in[0,t),\\
v_{s-t}(\theta_t),\ s\geq t.
\end{array}
\right.
\end{equation*}
Then, by a standard argument (see, e.g., \cite{Peng 1997}),
\begin{equation*}
  V_\lambda(X_t^{x,u})=\essinf\limits_{v\in\mathcal{U}}Y_t^{\lambda,t,X_t^{x,u},\overline{v}}=\essinf\limits_{v\in\mathcal{U}}Y_t^{\lambda,x,u\oplus\overline{v}},
\end{equation*}
where
\begin{equation*}
(u\oplus\overline{v})_s=
\left\{
\begin{array}{lll}
u_s,\ s\in[0,t)\\
\overline{v}_s,\ s\geq t
\end{array}
\right.
\in\mathcal{U}.
\end{equation*}
Again from an argument by now standard (see, e.g., \cite{Peng 1997}), for all $\varepsilon>0$, there exists $v\in\mathcal{U}$ such that $u=v$, dsd$\mathbb{P}$-a.e. on $[0,t]\times\Omega$ and
\begin{equation*}
 V_\lambda(X_t^{x,u})\geq Y_t^{\lambda,x,v}-\varepsilon,\ \mathbb{P}\text{-a.s.}
\end{equation*}
Then, from the monotonicity and the Lipschitz property (in $L^2$) of $G_{0,t}^{\lambda,x,u}[\cdot]$ (resulting from BSDE standard estimates, see, e.g., \cite{Peng 1997}),
\begin{equation*}
  \begin{split}
  G_{0,t}^{\lambda,x,u}[V_\lambda(X_t^{x,u})]\geq& G_{0,t}^{\lambda,x,u}[Y_t^{\lambda,x,v}-\varepsilon]
  \geq G_{0,t}^{\lambda,x,u}[Y_t^{\lambda,x,v}]-C\varepsilon
  =Y_0^{\lambda,x,v}-C\varepsilon\\
  \geq&\inf\limits_{v\in\mathcal{U}}Y_0^{\lambda,x,v}-C\varepsilon=V_\lambda(x)-C\varepsilon,\ \mathbb{P}\text{-a.s.}
  \end{split}
\end{equation*}
Consequently, letting $\varepsilon\downarrow0$, we see that
\begin{equation*}
  \overline{V}_\lambda(x)=\inf\limits_{u\in\mathcal{U}}G_{0,t}^{\lambda,x,u}[V_\lambda(X_t^{x,u})]\geq V_\lambda(x).
\end{equation*}
2) To prove that $V_\lambda(x)\geq\overline{V}_\lambda(x)$, we let, for any given $\varepsilon>0$, $u\in\mathcal{U}$ be such that $V_\lambda(x)\geq Y_0^{\lambda,x,u}-\varepsilon$. Then,
\begin{equation*}
  \begin{split}
  V_\lambda(x)\geq& Y_0^{\lambda,x,u}-\varepsilon=G_{0,t}^{\lambda,x,u}[Y_t^{\lambda,x,u}]-\varepsilon\\
  \geq&G_{0,t}^{\lambda,x,u}[\essinf\limits_{\overline{v}\in\mathcal{U}}Y_t^{\lambda,x,u\oplus\overline{v}}]-\varepsilon\\
  =&G_{0,t}^{\lambda,x,u}[\essinf\limits_{\overline{v}\in\mathcal{U}}Y_t^{\lambda,t,X_t^{x,u},\overline{v}}]-\varepsilon,\ \mathbb{P}\text{-a.s.}
  \end{split}
\end{equation*}
But, $Y_t^{\lambda,t,X_t^{x,u},\overline{v}}=(Y_0^{\lambda,y,v})(\theta_t)|_{y=X_t^{x,u}}$, and thus
\begin{equation*}
  \essinf\limits_{\overline{v}\in\mathcal{U}}Y_t^{\lambda,t,X_t^{x,u},\overline{v}}=(\inf\limits_{\overline{v}\in\mathcal{U}}Y_0^{\lambda,y,\overline{v}})(\theta_t)|_{y=X_t^{x,u}}=V_\lambda(X_t^{x,u}).
\end{equation*}
Consequently,
\begin{equation*}
  V_\lambda(x)\geq G_{0,t}^{\lambda,x,u}[V_\lambda(X_t^{x,u})]-\varepsilon,\ \text{for\ all}\ u\in\mathcal{U},
\end{equation*}
from where it follows that
\begin{equation*}
  V_\lambda(x)\geq \inf\limits_{u\in\mathcal{U}}G_{0,t}^{\lambda,x,u}[V_\lambda(X_t^{x,u})]-\varepsilon,
\end{equation*}
and letting $\varepsilon\downarrow0$ we get $V_\lambda(x)\geq\overline{V}_\lambda(x)$. \end{proof}


 \end{document}